\documentclass[review]{elsarticle}

\usepackage{a4wide,amssymb,graphics,graphicx,textcomp}
\usepackage{enumerate,textcomp,multirow,amsmath}
\usepackage{algorithm}
\usepackage{algorithmic}
\usepackage{url}
\usepackage{amsmath, amsthm}
\usepackage[colorlinks]{hyperref}
\usepackage{mathrsfs}
\numberwithin{equation}{section}

\newtheorem{theorem}{\bf Theorem}[section]

\newtheorem{definition}[theorem]{\bf Definition}

\newtheorem{lemma}[theorem]{\bf Lemma}
\newtheorem{example}[theorem]{\bf Example}

\def\real{\mathop{\mathrm{Re}}}

\newcommand {\mat}  [1] {\left[\begin{array}{#1}}
\newcommand {\rix}      {\end{array}\right]}

\newcommand{\eproof}{\space
    {\ \vbox{\hrule\hbox{\vrule height1.3ex\hskip0.8ex\vrule}\hrule}}\par}

\def\diag{\mathop{\mathrm{diag}}}
\def\rank{\mathop{\mathrm{rank}}}

\newcommand{\nrm}[1]{{\left\vert\kern-0.25ex\left\vert\kern-0.25ex\left\vert #1 
    \right\vert\kern-0.25ex\right\vert\kern-0.25ex\right\vert}}

\newcommand{\C}{{\mathbb C}}
\newcommand{\R}{{\mathbb R}}

\journal{a journal}

\begin{document}
\begin{frontmatter}

\title{Structured eigenvalue backward errors of Rosenbrock systems and related $\mu$-value problems}


\author{Anshul Prajapati\fnref{iitd}}
\author{Punit Sharma\fnref{iitd}}
\fntext[iitd]{Department of Mathematics, Indian Institute of Technology Delhi, Hauz Khas, 110016, India; \texttt{\{maz198078, punit.sharma\}@maths.iitd.ac.in.}
A.P. acknowledges the support of the CSIR Ph.D. grant by Ministry of Science \& Technology, Government of India. P.S. acknowledges the support of the SERB- CRG grant (CRG/2023/003221) and SERB-MATRICS grant by Government of India.}

\begin{abstract}
	In this paper, we compute the structured eigenvalue backward error of a Rosenbrock system matrix $S(z)=\mat{cc} A-zI & B \\ C & P(z) \rix$ for a given scalar $\lambda\in \C$. 
	We have developed simplified formulas for the structured eigenvalue backward error of the Rosenbrock system matrix, considering both full and partial block perturbations. These formulas involve computing 
	structured $\mu$-values of a rectangular matrix under rectangular-block-diagonal perturbations.
	 For the reformulated  $\mu$-value problem, we provide an explicit expression using partial isometric matrices and also obtain a computable upper bound, which is equal to the $\mu$-value when the pertrubation matrix has no more than three blocks at the diagonal.
	 The results are illustrated through numerical experiments.
\end{abstract}


\end{frontmatter}


\section{Introduction}
Consider a standard linear time-invariant system
\begin{equation}
	\Sigma : \quad
	\begin{array}{rl}
		\dot{x}(t)&=Ax(t)+Bu(t)\\
		y(t)&=Cx(t)+P(\frac{d}{dt})u(t)
	\end{array},
\end{equation} 
where $\frac{d}{dt}$ is the differential operator, $A\in \C^{r,r}, B\in \C^{r,n}, C\in \C^{n,r}$ and $P(z)=\sum_{k=0}^d z^k A_k$ with $A_k\in \C^{n,n}$ is a matrix polynomial of degree $d$, $u(t)$ is the input vector, $x(t)$ is the state vector, and $y(t)$ is the output vector. The matrix polynomial given by 
\begin{equation}\label{mat:system}
	S(z)=\mat{cc} A-z I_r & B \\ C & P(z) \rix,
\end{equation}
is called as the \emph{Rosenbrock system matrix} associated with the LTI system $\Sigma$. A scalar $\lambda \in \C$ is said to be an \emph{eigenvalue} of $S(z)$ if $\det(S(\lambda)) = 0$. The eigenvalues of $S(z)$ are the invariant zeros of the LTI system $\Sigma$. In general, the transmission zeros of the LTI system $\Sigma$ are a subset of the invariant zeros. However, if the system $\Sigma$ is minimal or irreducible, meaning that $\rank(\mat{cc}A-z I_r & B\rix) =r$ and $\rank(\mat{cc}A^T-zI_r & C^T \rix) = r$ for all $z\in \C$, then the transmission zeros are the same as the invariant zeros. The location of these zeros is a critical factor in system theory, as it determines the system's behaviour. As if all the poles of the system $\Sigma$ are on the left half of the complex plane, then the system is stable. This motivates to consider the perturbation analysis of eigenvalues of the system matrix $S(z)$. 

The system matrix $S(z)$ can also arise from the realization of a rational matrix function. Specifically, suppose we have an $n\times n$ rational matrix function $G(z)$. A tuple $(A,B,C,P(z))$, with $A\in \C^{r,r},B\in \C^{r,n}, C\in \C^{n,r}$ and $P(z)$ an $n\times n$ matrix polynomial of degree $d$ is called a realization of $G(z)$ if
\begin{equation}
	G(z)=P(z)+C(zI_r-A)^{-1}B.
\end{equation}
Rosenbrock~\cite{Ros70}, showed that if the associated system matrix $S(z)$ has the least order, that is, it is minimal or irreducible, then the finite poles of $G(z)$ are exactly the finite zeros (eigenvalues) of $A-zI_r$, and the finite zeros of $G(z)$ are exactly the finite zeros of $S(z)$.

Backward error plays an important role in the stability analysis of algorithms. For matrix polynomials, eigenvalue/eigenpair backward errors have been extensively studied in the literature, see~\cite{MR2496422,MR2780396,MR3194659,MR3335496,MR4404572,Tis00}. However, little work has been done to exploit backward errors of rational matrix functions. In~\cite{MR4098788}, authors have studied the eigenpair backward errors of rational matrix functions, while in~\cite{PraS22b} the eigenvalue backward errors of rational matrix functions under structure-preserving perturbations were considered. 
Recently in~\cite{LuPSB25}, authors studied the eigenvalue backward error of Rosenbrock system matrix $S(z)$ when the perturbations are measured with respect to norm
\begin{equation}\label{eq:nrm}
	\nrm{S(s)}_F:=\sqrt{{\|A\|}_F^2+{\|B\|}_F^2+{\|C\|}_F^2+{\| {A_0}\|}_F^2+\cdots +{\| {A_d}\|}_F^2},
\end{equation}
where ${\|\cdot\|}_F$ stands for the Frobenius norm. The authors reformulated the backward error problems into a class of optimization problems involving sum of two generalized Rayleigh quotients, for which they proposed a solution using a Nonlinear Eigenvalue Problem with eigenvector dependency. 

In this paper, we consider the eigenvalue backward error of the Rosenbrock system matrix under various types of block perturbations to $S(z)$ with respect to the max norm, i.e., 
\begin{equation}
	\nrm{ S(z)}_\infty = {\max}\left\{ \|A\|, \|B\|, \|C\|, \|{A_0}\|, \ldots, \|{A_d}\| \right\},
\end{equation}
where $\|\cdot\|$ stands for the operator 2-norm or spectral norm of a matrix. This norm was used in the groundbreaking paper by Tisseur~\cite{Tis00} to study backward error and condition of polynomial eigenvalue problem. 

We note that the framework suggested in~\cite{LuPSB25} for $\nrm{S(s)}_F$ does not extend to the eigenvalue backward error of $S(z)$ for $\nrm{ S(z)}_\infty$. To tackle this, we use the structured $\mu$-value of a matrix.
The \emph{structured $\mu$-value} (a.k.a. \emph{structured singular value})  of $M\in \C^{k,p}$ with respect to the perturbation set $\mathcal{S}\subseteq \C^{p.k}$ is defined as
\begin{equation}\label{def:mu1}
	\mu_{\mathcal{S}}(M)=\left( \min \{\|\Delta\|\; : \; \Delta \in \mathcal{S} \text{ and } {\det}(I_p-\Delta M)=0\}\right)^{-1},
\end{equation}
where $I_p$ denotes the identity matrix of size $p\times p$. If there does not exist $\Delta\in \mathcal{S}$ such that ${\det}(I_p-\Delta M)=0$, then $\mu_{\mathcal{S}}(M)=0$. The $\mu$-value problems occur in stability analysis of uncertain systems and in the eigenvalue perturbation theory of matrices and matrix polynomials~\cite{Kar03,ZhoDG96,doyle82}.

The aim of this paper is two fold: i) to tackle the eigenvalue backward error problem of $S(z)$ for $\nrm{ S(z)}_\infty$, reduce it to an equivalent problem of computing the structured $\mu$-value of a rectangular matrix under rectangular-block-diagonal perturbations, and ii) to revisit the structured $\mu$-value for a rectangular matrix $M\in \C^{k, p}$ under the structured perturbation class 
\begin{equation}\label{eq:pertset}
	\mathcal{S}=\big\{\diag(\Delta_1,\Delta_2,\ldots,\Delta_r)\; : \; \Delta_i\in \C^{p_i,k_i} \text{ s.t. } \sum_{i=1}^r p_i =p, \sum_{i=1}^r k_i=k \big\},
\end{equation}
and generalize the $\mu$-value results in~\cite{doyle82} , where
the diagonal blocks $\Delta_i$'s in~\eqref{eq:pertset} were all considered to be square. 
We note, however, that this generalization is not immediate and requires a careful dealing of the rectangular diagonal blocks. 

The paper is organized as follows: 
In Section~\ref{sec:ebrr}, we compute the eigenvalue backward error of a Rosenbrock matrix considering both full and partial block perturbations to $S(z)$. Here, we develope a simplified characterization of the eigenvalue backward error in terms of computing structured $\mu$-value of an appropriate matrix. All the results pertaining to the structured $\mu$-value of a rectangular matrix under rectangular-block-diagonal perturbations are presented in Section~\ref{sec:muvalue}. 
Here, we generalize the results in~\cite{doyle82} for $\mu_{\mathcal S}(M)$, when $M$ is rectangular and $\mathcal S$ is defined by~\eqref{eq:pertset}. 
In Section~\ref{sec:num}, we illustrate the results obtained over some numerical examples.

\section{Eigenvalue backward error of Rosenbrock system matrix}\label{sec:ebrr}

In this section, we consider the eigenvalue backward error of Rosenbrock system matrix $S(z)$ of the form~\eqref{mat:system} with respect to various types of perturbations. We consider perturbations in block matrices $A$, $B$, $C$ and $P(z)$ of $S(z)$, respectively, as $\Delta_A \in \C^{r,r}$, $\Delta_B\in \C^{r,n}$, $\Delta_C\in \C^{n,r}$ and $\Delta_P(z) \in \C[z]^{n,n}$, where $\Delta_P(z)=\sum_{j=0}^d z^j\Delta_{A_j}$ with $\Delta_{A_j} \in \C^{n,n}$ for $j=0,\ldots,d$, and define the set of perturbations in all blocks of $S(z)$ as
\begin{equation}\label{def:pertset}
	\mathbb{S}(A,B,C,P):= \left\{\Delta S(z) \:\; : \; \Delta S(z)=\mat{cc} \Delta_A & \Delta_B\\ \Delta_C & \Delta_P \rix \; \text{where} \; \Delta_P(z)= \sum_{j=0}^d z^j\Delta_{A_j} \right\}.
\end{equation}
To measure the perturbations to the Rosenbrock system matrix $S(z)$, we consider the following norm on $\Delta S(z) \in \mathbb{S}(A,B,C,P)$:
\begin{equation}
	{\nrm{\Delta S(z)}}_\infty = \max\left\{ \|\Delta_A\|, \|\Delta_B\|, \|\Delta_C\|, \|\Delta_{A_0}\|, \ldots, \|\Delta_{A_d}\| \right\}.
\end{equation}
The structured eigenvalue backward error of an approximate eigenvalue of $S(z)$ with respect to structured perturbations from the set $\mathbb{S}(A,B,C,P)$ is defined as follows.
\begin{definition}\label{def:eigerror}
	Let $S(z)$ be a Rosenbrock system matrix of the form~\eqref{mat:system} and let $\lambda \in \C$. Then, 
	\begin{equation}\label{def:error}
		\eta^{\mathbb{S}}(\lambda,A,B,C,P):= \inf \left\{ {\nrm{\Delta S(\lambda)}}_\infty \; : \; \Delta S(\lambda) \in \mathbb{S}(A,B,C,P), \; {\det}(S(\lambda)-\Delta S(\lambda))=0 \right\}
	\end{equation}
	is called the structured eigenvalue backward error of $S(z)$ for $\lambda$ with respect to perturbations to the blocks $A,B,C$ and $P(z)$ of $S(z)$.
\end{definition}
Similarly, we can define the eigenvalue backward errors of $S(z)$ while giving partial perturbations to the blocks of $S(z)$. These cases of partial perturbations to $S(z)$ are discussed in Section~\ref{sec:parperbb}.

We use the following notaton throughout the paper. For a matrix $M \in \C^{k, p}$, $\sigma_{\max}(M)$ denotes the largest singular value of $M$ and $\rho(A)$ denotes the spectral radius of a square matrix $A$.  The zero matrix of size $m \times n $ is denoted by $0_{m,n}$ and the zero matrix of size $m \times m $ is denoted by $0_{m}$.  
Our aim is to reformulate the eigenvalue backward error problem~\eqref{def:error} into an equivalent problem of computing structured $\mu$-value~\eqref{def:mu1}  of some related matrix and then apply the $\mu$-value results obtained in Section~\ref{sec:muvalue} to estimate the eigenvalue backward error. The following two matrices 
$\tilde{J_1}\in\C^{(r+n)d,nd}$ and $\tilde{J_2}\in \C^{nd,(r+n)}$ will be frequently used in achieving this. 
\begin{equation}\label{mat:tildeJ1J2}
	\tilde{J_1}:=\mat{cccc}0_{r,n}&&& \\I_n&&&\\&0_{r,n}&& \\ &I_n&& \\ &&\ddots&\\ &&&0_{r,n} \\&&&I_n\rix \quad \text{and} \quad \tilde{J_2}:=\mat{cc}0_{n,r}&I_n \\0_{n,r}&I_n \\ \vdots & \vdots  \\0_{n,r}&I_n  \rix .
\end{equation}  

\subsection{Backward error with full perturbation}\label{subsec:full}

Consider the Rosenbrock system matrix $S(z)$ in the form~\eqref{mat:system}. To obtain, the eigenvalue backward error of $S(z)$ for a given $\lambda \in \C$, we consider the perturbation to the matrices $A$, $B$, $C$ and $P(z)$ from the set $\mathbb S(A,B,C,P)$ defined by~\eqref{def:pertset}. The corresponding eigenvalue backward error $\eta^{\mathbb{S}}(\lambda,A,B,C,P)$ is defined by~\eqref{def:error}.

The eigenvalue backward error $\eta^{\mathbb{S}}(\lambda,A,B,C,P)$ is always finite, as the perturbation 
$\Delta \tilde S(z)=\mat{cc}A & B \\ C & P(z)\rix$ will make the perturbed system matrix as a zero matrix. This implies that 
\[
\eta^{\mathbb{S}}(\lambda,A,B,C,P) \leq \max \left\{\|A\|,\|B\|,\|C\|,\|A_0\|,\ldots,\|A_d\| 
\right\}< \infty.
\]
Also $\eta^{\mathbb{S}}(\lambda,A,B,C,P)=0$, if $\lambda$ is an eigenvalue of $S(z)$. Thus, in the following, we assume that $\lambda$ is not an eigenvalue of $S(z)$, i.e., $(S(\lambda))^{-1}$ exists.

The system matrix $S(z)$ can be seen as a matrix polynomial with zero block structure of the form 
\begin{equation}\label{eq:zerosparsys}
S(z)=P_0+zP_1+z^2P_2+\cdots+z^dP_d,
\end{equation}
where
$ P_0=\mat{cc}A & B \\ C & A_0 \rix$,  $ P_1=\mat{cc}-I_r & 0 \\ 0 & A_1 \rix$,  and $P_i= \mat{cc}0 & 0 \\ 0 & A_i \rix$ for $i=2,\ldots,d$.
Thus computing the eigenvalue backward error of $S(z)$ is equivalent to finding the eigenvalue backward error of $S(z)$ as a matrix polynomial in the form~\eqref{eq:zerosparsys} preserving the zero block structures of the coefficient matrices $P_i$'s.
If the Rosenbrock system matrix is considered as a matrix polynomial~\eqref{eq:zerosparsys} and the block struture in the coefficient matrices $P_i$'s  is ignored, then the \emph{unstructured eigenvalue backward error} of $S(z)$ can be defined as
\begin{equation}\label{def:eigerror-unst}
	\eta(\lambda,S):= \inf \left\{ \max\{\|\Delta_{P_0}\|,\ldots,\|\Delta_{P_d}\|\}\; : \; \Delta_{P_k} \in \C^{r+n,r+n}, \; {\det}(\sum_{k=0}^d \lambda^k (P_k - \Delta_{P_k}))=0 \right\}.
\end{equation}
A computable formula for $\eta(\lambda,S)$ is given in~\cite{AhmA09} as
	\begin{equation*}
	\eta(\lambda,S)=\frac{\sigma_{\min}(S(\lambda))}{\|(1,\lambda,\lambda^2,\ldots,\lambda^d)\|_{1}},
\end{equation*}
	where $\|(1,\lambda,\lambda^2,\ldots,\lambda^d)\|_{1} = 1+|\lambda| +|\lambda^2| + \cdots + |\lambda^d|$ and $\sigma_{\min}(S(\lambda))$ is the smallest singular value of 
	$S(\lambda)$. 
We note that the norm considered in~\eqref{def:error} to obtain the structured eigenvalue backward error and the one considered in~\eqref{def:eigerror-unst} for unstructured eigenvalue backward error are different. Therefore we connot directly compare them.


The following result reduces the problem of computing the eigenvalue backward error~\eqref{def:error} to a structured $\mu$-value problem of the form~\eqref{def:mu1}.

\begin{theorem}\label{theorem:ABCP}
Consider the Rosenbrock system matrix $S(z)$ of the form~\eqref{mat:system} and let $\lambda\in \C$ be such that $(S(\lambda))^{-1}$ exists. Let $J_1\in\C^{(d+1)(r+n),(d+2)n+2r}$ and $J_2\in\C^{(d+2)n+2r,(r+n)}$ be defined as
\begin{equation}\label{mat:J1J2}
	J_1:=\mat{cccc|c}I_{r}&I_r&&&\\&&I_n&I_n& \\ \hline&&&&\tilde{J_1} \rix \quad \text{and}\quad
	J_2:=\mat{c} \begin{array}{cc}I_{r}&\\&I_n\\I_r&\\&I_n \end{array}\\ \hline \begin{array}{c}\tilde{J_2}\end{array} \rix,
\end{equation}
where $\tilde{J_1}$ and $\tilde{J_2}$ are defined by~\eqref{mat:tildeJ1J2}. Then,
\begin{equation}\label{error:ABCP}
\eta^{\mathbb{S}}(\lambda,A,B,C,P)=(\mu_{\mathcal{S}}(M))^{-1},
\end{equation}
where $\mu_{\mathcal{S}}(M)$ is the $\mu$-value defined by~\eqref{def:mu1} for $M=J_2(S(\lambda))^{-1}[I_{r+n}\;\lambda I_{r+n}\; \cdots \; \lambda^dI_{r+n}]J_1$ and $\mathcal{S}=\left\{{\diag}(\Delta_1,\Delta_2,\ldots,\Delta_{d+4})\; : \; \Delta_1\in\C^{r,r},\; \Delta_2\in \C^{r,n},\; \Delta_3\in \C^{n,r} ,\; \Delta_i\in \C^{n,n},\; i=4,\ldots,d+4\right\}$.
\end{theorem}
\begin{proof}
From~\eqref{def:error}, the eigenvalue backward error $\eta^{\mathbb{S}}(\lambda,A,B,C,P)$ is given by
\begin{equation}\label{error:ABCP1}
\eta^{\mathbb{S}}(\lambda,A,B,C,P):= \inf \left\{ {\nrm{\Delta S(\lambda)}}_{\infty} \; : \; \Delta S(z) \in \mathbb{S}(A,B,C,P), \; {\det}(S(\lambda)-\Delta S(\lambda))=0 \right\},
\end{equation}
where $\mathbb{S}(A,B,C,P)$ is defined by~\eqref{def:pertset}. Observe that for a given 
$\Delta S(z)= \mat{cc} \Delta_A & \Delta_B \\ \Delta_C & \sum_{j=0}^d\Delta_{A_j}  \rix \in \mathbb{S}(A,B,C,P)$, we have
\begin{eqnarray*}
{\det}(S(\lambda)-\Delta S(\lambda))=0 & \iff {\det}\bigg ( S(\lambda) - \mat{cccc}I_{r+n} & \lambda I_{r+n} & \cdots & \lambda^{d}I_{r+n} \rix \mat{cc}\Delta_A & \Delta_B\\ \Delta_C & \Delta_{A_0} \\ \hline 0_{r, r}&0_{r, n}\\0_{n, r}&\Delta_{A_1} \\ \hline \vdots & \vdots \\ \hline  0_{r, r}&0_{r, n}\\0_{n, r}&\Delta_{A_d} \rix \bigg)=0
\end{eqnarray*}
if and only if 
\begin{equation*}
	{\det}\Big(S(\lambda)-\mat{cccc}I_{r+n} & \lambda I_{r+n} & \cdots & \lambda^{d}I_{r+n} \rix J_1 \Delta J_2 \Big)=0, 
\end{equation*}
where $J_1$, $J_2$ are defined by~\eqref{mat:J1J2} and $\Delta\in\C^{(d+2)n+2r,(d+2)n+2r}$ is a block diagonal matrix defined as
\begin{equation}\label{delta:ABCP1}
	\Delta := {\diag}(\Delta_A,\Delta_B,\Delta_C,\Delta_{A_0},\ldots, \Delta_{A_d}),
\end{equation}
which can be equivalently written as 
\begin{equation}\label{det:ABCPPP}
	{\det}\left(I_{(d+2)n+2r}-\Delta J_2 S(\lambda)^{-1} [I_{r+n}\; \lambda I_{r+n} \; \cdots \; \lambda^{d}I_{r+n}] J_1\right)=0.
\end{equation}
Also note that 
\begin{equation}\label{delta:ABCP22}
	{\nrm{\Delta S(z)}}_{\infty} =\max \left\{\|\Delta_A\|,\|\Delta_B\|,\|\Delta_C\|,\|\Delta_{A_0}\|,\ldots,\|\Delta_{A_d}\| 
	\right\}=\|\Delta\| = \sigma_{\max}(\Delta).
\end{equation}
Thus using~\eqref{det:ABCPPP}-\eqref{delta:ABCP22} in~\eqref{error:ABCP1}, we get
\begin{equation}\label{eq:reform1}
	\eta^{\mathbb{S}}(\lambda,A,B,C,P)=\inf \left\{ \sigma_{\max}(\Delta)\; :\; \Delta S(z) \in \mathbb{S}(A,B,C,P),\;{\det}(I_{(d+2)n+2r}-\Delta M)=0 \right\},
\end{equation}
where $M=J_2 (S(\lambda))^{-1} [I_{r+n}\; \lambda I_{r+n} \; \cdots \; \lambda^{d}I_{r+n}] J_1$ and $\Delta$ is defined by~\eqref{delta:ABCP1} for  $\Delta S(z)= \mat{cc} \Delta_A & \Delta_B \\ \Delta_C & \sum_{j=0}^d\Delta_{A_j}  \rix \in \mathbb{S}(A,B,C,P)$.
By setting $\mathcal{S}=\{{\diag}(\Delta_1,\Delta_2,\ldots,\Delta_{d+4})\; : \; \Delta_1\in\C^{r,r},$ $\; \Delta_2\in \C^{r,n},\; \Delta_3\in \C^{n,r} ,\; \Delta_i\in \C^{n,n},\; i=4,\ldots,d+4\}$,  \eqref{eq:reform1} yields that 
\begin{eqnarray*}\label{eq:reform2}
	\eta^{\mathbb{S}}(\lambda,A,B,C,P)&=&\inf \left\{ \sigma_{\max}(\Delta)\; :\; \Delta  \in \mathcal{S},\;{\det}(I_{(d+2)n+2r}-\Delta M)=0 \right\} \\
	&=& (\mu_{\mathcal{S}}(M))^{-1}.
\end{eqnarray*}
This completes the proof.
\end{proof}


\subsection{Backward error with partial perturbation}\label{sec:parperbb}

In this section, we consider eigenvalue backward error while considering partial perturbations to the blocks $A,B,C$, and $P(z)$ of $S(z)$. We divide the partial perturbations in three types (i)~perturbing only one of the blocks from $A,B,C$, and $P(z)$ of $S(z)$; (ii)~perturbing any two blocks from $A,B,C$, and $P(z)$ of $S(z)$, i.e., \{$A$ and $B$\}, \{$A$ and $C$\}, \{$A$ and $P(z)$\}, \{$B$ and $C$\}, \{$B$ and $P(z)$\} and \{$C$ and $P(z)$\};
(iii)~perturbing any three blocks  from $A,B,C$, and $P(z)$ of $S(z)$, i.e., \{$A$, $B$ and $C$\}, \{$A$, $B$ and $P(z)$\}, \{$A$, $C$ and $P(z)$\} and \{$B$, $C$ and $P(z)$\}.

Analogous to Section~\ref{subsec:full}, we define the eigenvalue backward error of $\lambda$ for $S(z)$ when only specific blocks from $A,B,C$, and $P(z)$ of $S(z)$ are subject to perturbation. This can be achieved by restricting the perturbation set
in~\eqref{def:pertset} and then define the corresponding backward error as in~\eqref{def:error}. 
For example, if we consider the case when only block $A$ of $S(z)$ is subject to perturbation, then the corresponding perturbation set in~\eqref{def:pertset} would be denoted by $\mathbb S (A)$ and defined as 
\[
\mathbb{S}(A):=\mathbb{S}(A,0,0,0)=\left\{ \Delta S(z) \; : \; \Delta_A \in \C^{r,r},\;\Delta S(z)=\mat{cc}\Delta_A &0\\0 & 0 \rix \right\}.
\]
In this case, the eigenvalue backward error of $S(z)$ for a given $\lambda \in \C$ with respect to perturbations from the set $\mathbb S(A)$ will be denoted by $\eta^{\mathbb{S}}(\lambda,A)$ and defined as 
\begin{equation}\label{eq:error_delA}
	\eta^{\mathbb{S}}(\lambda,A):=\inf\left\{ {\nrm{\Delta S(\lambda)}}_{\infty}\; : \; \Delta S(z)\in \mathbb{S}(A),\; {\rm det}(S(\lambda)-\Delta S(\lambda))=0 \right\}.
\end{equation}

\subsubsection{Perturbing only one block of $S(z)$ at a time}\label{sec:perone}

Here, we consider the backward error when we allow perturbation only in one of the blocks from $A$, $B$, $C$, and $P(z)$ of $S(z)$. The backward error is denoted by $\eta^{\mathbb{S}}(\lambda,A)$ and defined by~\eqref{eq:error_delA} when the perturbation is allowed only in block $A$  of $S(z)$. Similarly for other cases, we denote the backward error by $\eta^{\mathbb{S}}(\lambda,B),\eta^{\mathbb{S}}(\lambda,C)$ and $\eta^{\mathbb{S}}(\lambda,P)$. 

We note that unlike $\eta^{\mathbb{S}}(\lambda,A,B,C,P)$, the eigenvalue backward error $\eta^{\mathbb{S}}(\lambda,A)$ is not necessarily finite. For example, the Rosenbrock system matrix
\begin{equation}
	\tilde S(z)=\mat{c|ccc}a-z & 0 & 0 & 1\\ \hline 1 & 0 & 0 & 0 \\ 0 & 1 & 0 & 0 \\ 0 & 0 & 1 & 0 \rix ,\quad \text{where}~ a\in \C\setminus \{0\}.
\end{equation}
Then $\eta^{\mathbb{S}}(\lambda,A)=\infty$ for every $\lambda\in \C$, i.e., for any given $\lambda \in \C$, there does not exist $\Delta \tilde S(\lambda) \in \mathbb S(A)$ such that ${\det}(\tilde S(\lambda)-\Delta \tilde S(\lambda))=0$.
The following result obtains a computable formula for $\eta^{\mathbb{S}}(\lambda,A)$ when $\eta^{\mathbb{S}}(\lambda,A)$ is finite.
\begin{theorem}\label{thm:perA}
Consider the Rosenbrock system matrix $S(z)$ of the form~\eqref{mat:system} and let $\lambda\in \C$ be such that $(S(\lambda))^{-1}$ exists. Suppose that $\eta^{\mathbb{S}}(\lambda,A)< \infty$, then
\begin{equation}
\eta^{\mathbb{S}}(\lambda,A)=\frac{1}{\sigma_{\max}\left([I_r\; 0_{r,n}](S(\lambda))^{-1}[I_r\; 0_{r,n}]^T \right)}.
\end{equation} 
\end{theorem}
\begin{proof}
From~\eqref{eq:error_delA}, we have
\begin{equation}\label{error:A2}
\eta^{\mathbb{S}}(\lambda,A)=\inf\left\{ {\nrm{\Delta S(\lambda)}}_\infty\; : \; \Delta S(z)\in \mathbb{S}(A),\; {\rm det}(S(\lambda)-\Delta S(\lambda))=0 \right\}.
\end{equation}
For any $\Delta S(z)=\mat{cc}\Delta_A & 0\\0& 0\rix \in \mathbb S(A)$, the
determinant condition in~\eqref{error:A2} can be equivalently written as,
\begin{align*}\label{det:A}
{\det}(S(\lambda)-\mat{cc}\Delta_A & 0\\ 0& 0 \rix)=0 & \iff {\det}(S(\lambda) - [I_r\; 0_{r,n}]^T \Delta_A [I_r\; 0_{r,n}])=0 \\
& \iff {\det}(I_r - \Delta_A [I_r\; 0_{r,n}]S(\lambda)^{-1}[I_r\; 0_{r,n}]^T)=0,
\end{align*}
%
 which is equivalent to the condition that there exists $w\in \C^{r}\setminus \{0\}$ such that 
\begin{equation}\label{eq:perAdet}
\Delta_A [I_r\; 0_{r,n}](S(\lambda))^{-1}[I_r\; 0_{r,n}]^T w=w.
\end{equation} 
By setting $ H:=[I_r\; 0_{r,n}](S(\lambda))^{-1}[I_r\; 0_{r,n}]^T$ and using~\eqref{eq:perAdet} for the determinant condition in~\eqref{error:A2}, we obtain
%
\begin{eqnarray}\label{error:perA3}
\eta^{\mathbb{S}}(\lambda,A)&=&\inf\left\{\|\Delta_A\| \; : \;\Delta_A \in \C^{r,r},\;  w\in \C^r\setminus \{0\},\; \Delta_A H w=w \right\} \nonumber\\
&=& \inf\left\{\|\Delta_A\| \; : \;\Delta_A \in \C^{r,r},\;  w\in \C^r,\; Hw\neq 0,\;  \Delta_A H w=w \right\}.
\end{eqnarray}
Since $\eta^{\mathbb{S}}(\lambda,A)$ is assumed to be finite, we have that $\text{null}(H) \neq \C^r$ and thus for any $w \in \C^r$ such that $Hw \neq 0$, 
from~\cite{trenkler2004matrices}, there always exists $\Delta_A \in \C^{r,r}$ satisfying~\eqref{eq:perAdet}. The minimal spectral norm of such a $\Delta_A$ satisfies that 
\begin{equation}\label{norm:perA}
	\|\Delta_A\|=\frac{\|w\|}{\|H w\|}.
\end{equation}
Using~\eqref{norm:perA} in~\eqref{error:perA3}, we obtain,
\begin{eqnarray}\label{eq:perAsig}
\eta^{\mathbb{S}}(\lambda,A)&=&\inf \left\{\frac{\|w\|}{\|H w\|}\; : \; w\in \C^r ,\;H w \neq 0 \right\} \nonumber\\
&=&\inf \left\{\frac{\|w\|}{\|H w\|}\; : \; w\in \C^r\setminus\{0\} \right\} \\
&=& \frac{1}{\sigma_{\max}(H )},\nonumber
\end{eqnarray}
where in~\eqref{eq:perAsig} the condition $H w \neq 0$ was dropped because of the assumption that $\eta^{\mathbb{S}}(\lambda,A)<\infty$ as the infimum in~\eqref{eq:perAsig} will not be attained at vectors $w$ for which $Hw=0$. Thus including vectors $w$ for which $H w = 0$ will not affect the infimum in~\eqref{eq:perAsig}. This completes the proof.
\end{proof}
Similar to $\eta^{\mathbb{S}}(\lambda,A)$, we can define the eigenvalue backward error of $S(z)$ for a given $\lambda \in \C$ with respect to perturbations only to block $B$ or $C$ or $P(z)$ by restricting the perturbation sets to $\mathbb{S}(B):=\mathbb{S}(0,B,0,0)$ or $\mathbb{S}(C):=\mathbb{S}(0,0,C,0)$ or $\mathbb{S}(P):=\mathbb{S}(0,0,0,P)$, respectively. The corresponding backward errors are respectively denoted by $\eta^{\mathbb{S}}(\lambda,B)$, $\eta^{\mathbb{S}}(\lambda,C)$ and $\eta^{\mathbb{S}}(\lambda,P)$,  and obtained in the following result.  
\begin{theorem}\label{thm:oneperturbothercases}
Consider the Rosenbrock system matrix $S(z)$ of the form~\eqref{mat:system} and let $\lambda\in \C$ be such that $(S(\lambda))^{-1}$ exists.
\begin{enumerate}
\item[(i)] If $\eta^{\mathbb{S}}(\lambda,B)<\infty$, then we have
\begin{equation}
\eta^{\mathbb{S}}(\lambda,B)=\frac{1}{\sigma_{\max}([0_{n,r}\; I_n]{(S(\lambda))}^{-1}[I_r\; 0_{r,n}]^T)};
\end{equation}
\item[(ii)] if $\eta^{\mathbb{S}}(\lambda,C)<\infty$, then we have
\begin{equation}
\eta^{\mathbb{S}}(\lambda,C)=\frac{1}{\sigma_{\max}([I_r\; 0_{r,n}]{(S(\lambda))}^{-1}[0_{n,r}\; I_n]^T)};
\end{equation}
\item[(iii)] if $\eta^{\mathbb{S}}(\lambda,P)<\infty$, then define $J_1:={\diag}(\mat{c}0_{r,n}\\I_n \rix,\tilde{J_1})$ and $J_2:=\mat{c}\begin{array}{cc}0_{n,r}&I_n\end{array} \\ \hline \tilde{J_2}\rix$, where $\tilde{J_1} \text{ and } \tilde{J_2}$ are defined by~\eqref{mat:tildeJ1J2}. Then we have
\begin{equation}
\eta^{\mathbb{S}}(\lambda,P)=\left(\mu_{\mathcal{S}}(M)\right)^{-1},
\end{equation}
where $\mu_{\mathcal{S}}(M)$ is defined by~\eqref{def:mu1}  for the matrix 
 $M:=J_2{(S(\lambda))}^{-1} [I_{r+n}\; \lambda I_{r+n}\; \cdots \;\lambda^{d}I_{r+n}]J_1$
  and  $\mathcal{S}=\{{\diag}(\Delta_1,\ldots,\Delta_{d+1})\; : \; \Delta_i\in \C^{n,n}, i=1,2,\ldots,d+1\}$. 
\end{enumerate}
\end{theorem}
\proof
The proofs of (i) and (ii) follow the proof of Theorem~\ref{thm:perA}, whereas the proof of (iii) is similar to Theorem~\ref{theorem:ABCP}.
\eproof

\subsubsection{Perturbing any two blocks of $S(z)$ at a time}\label{sec:pertwo}

In this section, we consider the eigenvalue backward error of a $\lambda \in \C$ for $S(z)$ while perturbing any two blocks from $A,B,C$, and $P(z)$ of $S(z)$. There are six such cases, i.e., perturbing $\{A \, \text{and}\, B\}$ or $\{A \, \text{and}\, C\}$ or $\{A \, \text{and}\, P(z)\}$ or $\{B \, \text{and}\, C\}$ or $\{B \, \text{and}\, P(z)\}$, or $\{C \, \text{and}\, P(z)\}$. Here, we give the details only for the case of perturbing blocks $\{A \, \text{and}\, B\}$ of $S(z)$ at a time. 
The results for the other cases of perturbing any two blocks follow similarly. 

In view of~\eqref{def:pertset}, the perturbation set consisting of perturbation matrices $\Delta S(z)$ with respect to perturbations only to $A$ and $B$ is denoted by $\mathbb{S}(A,B):=\mathbb{S}(A,B,0,0)$. The corresponding eigenvalue backward error for a given scalar $\lambda \in \C$ is denoted by $\eta^{\mathbb{S}}(\lambda,A,B)$ and is defined by~\eqref{def:error} when the perturbations are restricted to the set $\mathbb{S}(A,B)$, i.e., 
\begin{equation}\label{eq:error_delA_B}
	\eta^{\mathbb{S}}(\lambda,A,B):=\inf\left\{ {\nrm{\Delta S(\lambda)}}_\infty\; : \; \Delta S(z)\in \mathbb{S}(A,B),\; {\rm det}(S(\lambda)-\Delta S(\lambda))=0 \right\}.
\end{equation}
We note that unlike $\eta^{\mathbb{S}}(\lambda,A)$ or $\eta^{\mathbb{S}}(\lambda,B)$ of the previous section, the eigenvalue backward error $\eta^{\mathbb{S}}(\lambda,A,B)$ is always finite, since 
perturbations $\Delta_A=A-\lambda I_r$ and $\Delta_B=B$ make the perturbed $S(\lambda)-\Delta S(\lambda)$ singular. 
Thus, we have
\begin{equation*}
	\eta^{\mathbb{S}}(\lambda,A,B)\leq \max\{\|A\|,\|B\|\}< \infty.
\end{equation*} 
The other backward errors, i.e., $\eta^{\mathbb{S}}(\lambda,A,C)$, $\eta^{\mathbb{S}}(\lambda,A,P)$, $\eta^{\mathbb{S}}(\lambda,B,C)$, $\eta^{\mathbb{S}}(\lambda,B,P)$, and $\eta^{\mathbb{S}}(\lambda,C,P)$ are defined analogously. 
\begin{theorem}\label{thm:perAB}
Consider the Rosenbrock system matrix $S(z)$ of the form~\eqref{mat:system} and let $\lambda\in \C$ be such that $(S(\lambda))^{-1}$ exists. Then the eigenvalue backward error while considering perturbation to blocks $A$ and $B$ of $S(z)$ is given by
\begin{equation}
\eta^{\mathbb{S}}(\lambda,A,B)=\left(\mu_{\mathcal{S}}\left(M\right) \right)^{-1},
\end{equation}
where $\mu_{\mathcal{S}}(M)$ is defined by~\eqref{def:mu} for 
$\mathcal{S}=\left\{{\diag}(\Delta_1,\Delta_2)\: :\; \Delta_1\in\C^{r,r}, \;\Delta_2\in \C^{r,n} \right\}$ and the matrix $M=(S(\lambda))^{-1}\mat{cc}I_r & I_{r}\\ 0_{n,r} & 0_{n,r} \rix$.
\end{theorem}
\begin{proof}
	In view of~\eqref{eq:error_delA_B}, for any $\Delta S(z)=\mat{cc}\Delta_A & \Delta_B \\ 0 & 0\rix$ where $\Delta_A \in \C^{r,r}$ and $\Delta_B \in \C^{r,n}$, we have
\begin{align*}
{\det}(S(\lambda)-\Delta S(\lambda))=0 & \iff {\det}\left(S(\lambda)-\mat{cc}I_r & I_r \\ 0_{n,r} & 0_{n,r} \rix \mat{cc}\Delta_A & \\ & \Delta_B\rix\right)=0,
\end{align*} 
which is equivalent to 
\begin{equation}\label{det:AB}
{\det}\left(I_{2r} - \mat{cc}\Delta_A & \\ & \Delta_B \rix (S(\lambda))^{-1}\mat{cc}I_r & I_r \\ 0_{n,r} & 0_{n,r} \rix \right)=0.
\end{equation}
Using~\eqref{det:AB} for the determinant condition in~\eqref{eq:error_delA_B}, we have
\begin{eqnarray*}
\eta^{\mathbb{S}}(\lambda,A,B)&=&\inf \Big\{\|\Delta\|\; : \;\Delta_A \in \C^{r,r},\;\Delta_B \in \C^{r,n},\; \Delta ={\diag} (\Delta_A,\Delta_B), \; \\
&& \hspace {2cm}{\det}(I_{2r}-\Delta (S(\lambda))^{-1}\mat{cc}I_r & I_r \\ 0_{n,r} & 0_{n,r} \rix )=0 \Big\}\\
&=&\mu_{\mathcal{S}}\Big((S(\lambda))^{-1}\mat{cc}I_r & I_r \\ 0_{n,r} & 0_{n,r} \rix\Big),
\end{eqnarray*}
where $\mathcal{S}=\left\{{\diag}(\Delta_1,\Delta_2)\: :\; \Delta_1\in\C^{r,r}, \;\Delta_2\in \C^{r,n} \right\}$. This completes the proof.
\end{proof}

In the follwoing, we state a result for the backward errors
$\eta^{\mathbb{S}}(\lambda,A,C)$, $\eta^{\mathbb{S}}(\lambda,A,P)$, $\eta^{\mathbb{S}}(\lambda,B,C)$, $\eta^{\mathbb{S}}(\lambda,B,P)$, and $\eta^{\mathbb{S}}(\lambda,C,P)$ proof of which is analogous to Theorem~\ref{thm:perAB} and hence omitted. 

\begin{theorem}\label{theorem:two}
Consider the Rosenbrock system matrix $S(z)$ of the form~\eqref{mat:system} and let $\lambda\in \C$ be such that $(S(\lambda))^{-1}$ exists. Let $\tilde{J_1}$ and $\tilde{J_2}$ be as defined in~\eqref{mat:tildeJ1J2}. Then,
\begin{enumerate}
\item perturbing both $A$ and $C$, i.e. when the perturbation set is $\mathbb S (A,C):=\mathbb S(A,0,C,0)$ in~\eqref{def:error}, we have
\begin{equation}
\eta^{\mathbb{S}}(\lambda,A,C)=\Big(\mu_{\mathcal{S}}\big(\mat{cc}I_r & 0_{r,n}\\I_r & 0_{r,n}\rix (S(\lambda))^{-1} \big)\Big)^{-1},
\end{equation}
where $\mathcal{S}=\{{\diag}(\Delta_1,\Delta_2)\;:\; \Delta_1\in\C^{r,r}, \; \Delta_2\in\C^{n,r}\}$.
\item perturbing both $A$ and $P(z)$, i.e. when the perturbation set is $\mathbb S (A,P):=\mathbb S(A,0,0,P)$ in~\eqref{def:error}, we have
\begin{equation}
\eta^{\mathbb{S}}(\lambda,A,P)=\left(\mu_{\mathcal{S}}(J_2(S(\lambda))^{-1} [I_{r+n}\; \lambda I_{r+n}\; \cdots \;\lambda^{d}I_{r+n}]J_1)\right)^{-1},
\end{equation}
where $\mathcal{S}=\{{\diag}(\Delta_1,\Delta_2,\ldots,\Delta_{d+2})\; : \; \Delta_1\in \C^{r,r},\;\Delta_i\in \C^{n,n}, i=2,\ldots,d+2\}$, and $J_1\in \C^{(d+1)(r+n),(d+1)n+r}$ and $J_2\in \C^{(d+1)n+r,(r+n)}$ are defined as
\begin{equation}
J_1=\mat{cc|c} I_r && \\&I_n&\\ \hline &&\tilde{J_1}\rix \quad \text{and}\quad  J_2=\mat{c}\begin{array}{cc}I_r& \\ & I_n \end{array}\\ \hline \tilde{J_2} \rix.
\end{equation}
\item perturbing both $B$ and $C$, i.e. when the perturbation set is $\mathbb S (B,C):=\mathbb S(0,B,C,0)$ in~\eqref{def:error}, we have
\begin{equation}
\eta^{\mathbb{S}}(\lambda,B,C)=\left( \mu_{\mathcal{S}}\left(\mat{cc} &I_n\\I_r& \rix (S(\lambda))^{-1}\right) \right)^{-1},
\end{equation}
where $\mathcal{S}=\{{\diag}(\Delta_1,\Delta_2)\;:\; \Delta_1\in\C^{r,n}, \; \Delta_2\in\C^{n,r}\}$.
 \item perturbing both $B$ and $P(z)$, i.e. when the perturbation set is $\mathbb S (
 B,P):=\mathbb S(0,B,0,P)$ in~\eqref{def:error}, we have
\begin{equation}
\eta^{\mathbb{S}}(\lambda,B,P)=\left(\mu_{\mathcal{S}}(J_2(S(\lambda))^{-1} [I_{r+n}\; \lambda I_{r+n}\; \cdots \;\lambda^{d}I_{r+n}]J_1)\right)^{-1},
\end{equation}
where $\mathcal{S}=\{{\diag}(\Delta_1,\Delta_2,\ldots,\Delta_{d+2})\; : \; \Delta_1\in\C^{r,n},\;\Delta_i\in \C^{n,n}, i=2,3,\ldots,d+2\}$ and $J_1\in \C^{(d+1)(r+n),(d+1)n+r}$ and $J_2\in \C^{(d+2)n,(r+n)}$ are defined as
\begin{equation}
J_1=\mat{cc|c} I_r && \\&I_n&\\ \hline &&\tilde{J_1}\rix, \quad J_2=\mat{c}\begin{array}{cc}0_{n,r}&I_n \\0_{n,r} & I_n \end{array}\\ \hline \tilde{J_2} \rix.
\end{equation}
\item perturbing both $C$ and $P(z)$, i.e. when the perturbation set is $\mathbb S (C,P):=\mathbb S(0,0,C,P)$ in~\eqref{def:error}, we have
\begin{equation}
\eta^{\mathbb{S}}(\lambda,C,P)=\left(\mu_{\mathcal{S}}(J_2(S(\lambda))^{-1} [I_{r+n}\; \lambda I_{r+n}\; \cdots \;\lambda^{d}I_{r+n}]J_1)\right)^{-1},
\end{equation}
where $\mathcal{S}=\{{\diag}(\Delta_1,\Delta_2,\ldots,\Delta_{d+2})\; : \; \Delta_1\in\C^{n,r},\;\Delta_i\in \C^{n,n}, i=2,3,\ldots,d+2\}$ and $J_1\in \C^{(d+1)(r+n),(d+2)n}$ and $J_2\in \C^{(d+1)n+r,(r+n)}$ are defined as
\begin{equation}
J_1=\mat{cc|c} 0_{r,n} &0_{r,n}& \\I_n&I_n&\\ \hline &&\tilde{J_1}\rix \quad \text{and} \quad J_2=\mat{c}\begin{array}{cc}I_r& \\ & I_n \end{array}\\ \hline \tilde{J_2} \rix.
\end{equation}
\end{enumerate}
\end{theorem}

\subsubsection{Perturbing any three blocks of $S(z)$ at a time}\label{sec:perthree}

In this section, we investigate the eigenvalue backward error of a $\lambda \in \C$ for $S(z)$ while perturbing any three blocks from $A$, $B$, $C$ and $P(z)$ of $S(z)$. There are four such cases, i.e., perturbing $\{A,B,C\}$ or  $\{A,B,P(z)\}$ or  $\{A,C, P(z)\}$, or  $\{B,C,P(z)\}$. 
Similar to Sections~\ref{sec:perone} and~\ref{sec:pertwo}, by constraining the perturbation set~\eqref{def:pertset} appropriately,~\eqref{def:error} defines the relevant backward errors. 
For example, the perturbation set consisting of perturbation matrices $\Delta S(z)$ with respect to perturbations only to $A$, $B$, and $C$ is denoted by $\mathbb{S}(A,B,C):=\mathbb{S}(A,B,C,0)$. The corresponding eigenvalue backward error for a given scalar $\lambda \in \C$ is denoted by $\eta^{\mathbb{S}}(\lambda,A,B,C)$ and is defined by~\eqref{def:error} when the perturbations are restricted to the set $\mathbb{S}(A,B,C)$, i.e., 
\begin{equation}\label{error:ABC}
	\eta^{\mathbb{S}}(\lambda,A,B,C):=\inf\left\{ \nrm{\Delta S(\lambda)}\; : \; \Delta S(z)\in \mathbb{S}(A,B,C),\; {\rm det}(S(\lambda)-\Delta S(\lambda))=0 \right\}.
\end{equation}
The other backward errors, i.e., $\eta^{\mathbb{S}}(\lambda,A,B,P)$, $\eta^{\mathbb{S}}(\lambda,A,C,P)$, and $\eta^{\mathbb{S}}(\lambda,B,C,P)$ are defined analogously. The backward error is finite in each of these cases. 
\begin{theorem}\label{thm:per_ABC}
Consider the Rosenbrock system matrix $S(z)$ of the form~\eqref{mat:system} and let $\lambda\in \C$ be such that $(S(\lambda))^{-1}$ exists. Then
\begin{equation*}
\eta^{\mathbb{S}}(\lambda,A,B,C)=\left(\mu_{\mathcal{S}}\left(M\right) \right)^{-1},
\end{equation*}
where $\mu_{\mathcal{S}}(M)$ is defined by~\eqref{def:mu} for 
$\mathcal{S}=\left\{{\diag}(\Delta_1,\Delta_2,\Delta_3)\: :\; \Delta_1\in\C^{r,r}, \;\Delta_2\in \C^{r,n},\; \Delta_3\in \C^{n,r} \right\}$ and 
$M:=\mat{cc}I_r & \\&I_n \\I_r&  \rix (S(\lambda))^{-1}\mat{ccc}I_r & I_r&\\&&I_n \rix$.
\end{theorem}
\proof The proof follows similar to Theorem~\ref{thm:perAB} by observing that for any 
$\Delta S(z)=\mat{cc}\Delta_A & \Delta_B \\ \Delta_C & 0\rix \in \mathbb S(A,B,C)$, where 
$\Delta_A \in \C^{r,r}$, $\Delta_B\in \C^{r,n}$, and $\Delta_C\in\C^{n,r}$, we have
\begin{align*}
	{\det}(S(\lambda)-\Delta S(\lambda))=0 & \iff {\det}\left(S(\lambda) - \mat{cc}\Delta_A & \Delta_B \\ \Delta_C & 0  \rix \right)=0 
\end{align*}
if and only if 
\[
{\det}\Big(S(\lambda) - \mat{ccc}I_r & I_r& 0_{r,n}\\ 0_{n,r}&0_{n,r}&I_n \rix{\diag}(\Delta_A,\Delta_B,\Delta_C)\mat{cc}I_r & 0_{r,n} \\ 0_{n,r}&I_n \\I _r& 0_{r,n} \rix \Big)=0,
\]
which is equivalent to the condition that 
\begin{equation}\label{det:ABC}
{\det}\left(I_{2r+n} - \Delta J_2 (S(\lambda))^{-1} J_1\right)=0,
\end{equation}
where $J_1=\mat{ccc}I_r && \\ &&I_n \rix$,  $J_2=\mat{cc}I_r &  \\ &I_n \\I _r&  \rix$ and $\Delta={\diag}(\Delta_A,\Delta_B,\Delta_C)$. 
Using~\eqref{det:ABC} for the determinant condition in~\eqref{error:ABC}, we have
\begin{align*}
	\eta^{\mathbb{S}}(\lambda,A,B,C)&=\inf\{\max\{\|\Delta_A\|,\|\Delta_B\|,\|\Delta_C\| \}\; : \; \Delta_A \in \C^{r,r}, \;\Delta_B\in \C^{r,n},\; \Delta_C\in\C^{n,r},\;&\\ 
	& \quad \quad \quad\Delta={\diag}(\Delta_A,\Delta_B,\Delta_C),\;{\det}(I_{2r+n} - \Delta J_2 S(\lambda)^{-1} J_1)=0\}\\
	&=\inf\{\|\Delta\|\; : \; \Delta \in \mathcal S,\; {\det}(I_{2r+n} - \Delta J_2 S(\lambda)^{-1} J_1)=0\}\\
	& =\left(\mu_{\mathcal{S}}(J_2(S(\lambda))^{-1}J_1) \right)^{-1},
\end{align*}
where $\mathcal S=\{ \Delta \;:\;\Delta_1 \in \C^{r,r}, \;\Delta_2\in \C^{r,n},\; \Delta_3\in\C^{n,r},\;\Delta={\diag}(\Delta_1,\Delta_2,\Delta_3)\}$ and $\mu_{\mathcal{S}}(\cdot)$ is defined by~\eqref{def:mu}.
\eproof
Next, we state a result for the backward errors $\eta^{\mathbb{S}}(\lambda,A,B,P)$, $\eta^{\mathbb{S}}(\lambda,A,C,P)$, and $\eta^{\mathbb{S}}(\lambda,B,C,P)$, proof of which is analogous to Theorem~\ref{thm:per_ABC} and hence omitted. 
\begin{theorem}\label{theorem:three}
	Consider the Rosenbrock system matrix $S(z)$ of the form~\eqref{mat:system} and let $\lambda\in \C$ be such that $(S(\lambda))^{-1}$ exists. Let $\tilde{J_1}$ and $\tilde{J_2}$ be defined by~\eqref{mat:tildeJ1J2}. Then,
	\begin{enumerate}
		\item perturbing $A$, $B$ and $P$, i.e. when the perturbation set is $\mathbb S (A,B,P):=\mathbb S(A,B,0,P)$ in~\eqref{def:error}, we have
\begin{equation}
\eta^{\mathbb{S}}(\lambda,A,B,P)=\left(\mu_{\mathcal{S}}(J_2(S(\lambda))^{-1} [I_{r+n}\; \lambda I_{r+n}\; \cdots \;\lambda^{d}I_{r+n}]J_1)\right)^{-1},
\end{equation}
where $\mathcal{S}=\{{\diag}(\Delta_1,\ldots,\Delta_{d+3})\; : \; \Delta_1\in\C^{r,r},\Delta_2\in\C^{r,n},\;\Delta_i\in \C^{n,n}, i=3,\ldots,d+3\}$ and $J_1\in \C^{(d+1)(r+n),(d+1)n+2r}$ and $J_2\in \C^{(d+2)n+r,(r+n)}$ are defined as
\begin{equation}
J_1=\mat{ccc|c} I_r &I_r&0_{r,n}& \\0_{n,r}&0_{n,r}&I_n&\\ \hline &&&\tilde{J_1}\rix \quad \text{and} \quad J_2=\mat{c}\begin{array}{cc}I_r& 0_{r,n} \\ 0_{n,r}& I_n\\0_{n,r}& I_n \end{array}\\ \hline \tilde{J_2} \rix.
\end{equation}
		\item perturbing $A$, $C$ and $P$, i.e. when the perturbation set is $\mathbb S (A,C,P):=\mathbb S(A,0,C,P)$ in~\eqref{def:error}, we have
\begin{equation}
\eta^{\mathbb{S}}(\lambda,A,C,P)=\left(\mu_{\mathcal{S}}(J_2(S(\lambda))^{-1} [I_{r+n}\; \lambda I_{r+n}\; \cdots \;\lambda^{d}I_{r+n}]J_1)\right)^{-1},
\end{equation}
where $\mathcal{S}=\{{\diag}(\Delta_1,\ldots,\Delta_{d+3})\; : \; \Delta_1\in\C^{r,r},\Delta_2\in\C^{n,r},\;\Delta_i\in \C^{n,n}, i=3,\ldots,d+3\}$ and $J_1\in \C^{(d+1)(r+n),(d+2)n+r}$ and $J_2\in \C^{(d+1)n+2r,(r+n)}$ are defined as
\begin{equation}
J_1=\mat{ccc|c} I_r &0_{r,n}&0_{r,n}& \\0_{n,r}&I_n&I_n&\\ \hline &&&\tilde{J_1}\rix \quad \text{and} \quad J_2=\mat{c}\begin{array}{cc}I_r&0_{r,n} \\ I_r&0_{r,n} \\0_{n,r}& I_n \end{array}\\ \hline \tilde{J_2} \rix.
\end{equation}
		\item perturbing $B$, $C$ and $P$, i.e. when the perturbation set is $\mathbb S (B,C,P):=\mathbb S(0,B,C,P)$ in~\eqref{def:error}, we have
\begin{equation}
\eta^{\mathbb{S}}(\lambda,B,C,P)=\left(\mu_{\mathcal{S}}(J_2(S(\lambda))^{-1} [I_{r+n}\; \lambda I_{r+n}\; \cdots \;\lambda^{d}I_{r+n}]J_1)\right)^{-1},
\end{equation}
where $\mathcal{S}=\{{\diag}(\Delta_1,\ldots,\Delta_{d+3})\; : \; \Delta_1\in\C^{r,n},\Delta_2\in\C^{n,r},\;\Delta_i\in \C^{n,n}, i=3,\ldots,d+3\}$, and $J_1\in \C^{(d+1)(r+n),(d+2)n+r}$ and $J_2\in \C^{(d+2)n+r,(r+n)}$ are defined as
\begin{equation}
J_1=\mat{ccc|c} I_r &0_{r,n}&0_{r,n}& \\0_{n,r}&I_n&I_n&\\ \hline &&&\tilde{J_1}\rix \quad \text{and} \quad J_2=\mat{c}\begin{array}{cc}0_{n,r}&I_n \\ I_r& 0_{r,n}\\0_{n,r}& I_n \end{array}\\ \hline \tilde{J_2} \rix.
\end{equation}
\end{enumerate}
\end{theorem}

\section{Structured $\mu$-values of rectangular matrices}\label{sec:muvalue}

Regardless of full and partial perturbations, the eigenvalue backward errors of the Rosenbrock system matrix can be expressed as the structured $\mu$-value of  a matrix 
$M \in \C^{k,p}$ in the form 
\begin{equation}\label{def:mu} 
	\mu_{\mathcal{S}}(M):= \left(\min \left\{ \sigma_{\max}(\Delta)\; : \; \Delta \in \mathcal{S}, \; \text{det}(I_p - \Delta M) = 0 \right\} \right)^{-1}
\end{equation}
with respect to the perturbation set 
\begin{equation}\label{eq:defpermu}
	\mathcal{S}:=\left\{ \diag(\Delta_1,\Delta_2,\ldots,\Delta_n)\; : \; \Delta_i \in \C^{p_i \times k_i},\; \sum_{i=1}^n k_i = k ,\; \sum_{i=1}^n p_i = p \right\}.
\end{equation}
For partial perturbations, a few cases (Theorem~\ref{thm:perA} and Theorem~\ref{thm:oneperturbothercases}) involve an exact formula in terms of the largest singular value of some matrix, but in general we must address the $\mu$-value problem~\eqref{def:mu}. 

When $\mathcal S=\C^{p,k}$ in~\eqref{def:mu}, then $\mu_{\mathcal{S}}(M)$ is called the unstructured $\mu$-value and denoted by $\mu(M)=\mu_{\C^{p,k}}(M)$.
It is shown in~\cite{Kar11} that 
\[
\mu(M)=\|M\|=\sigma_{\max}(M).
\]
Clearly, we have 
\[
\mu_{\mathcal{S}}(M)\leq \mu(M)=\|M\|.
\]

The $\mu$-value problem for a square matrix $M \in \C^{n,n}$, subject to the perturbation set $\tilde {\mathcal S}=\left\{{\diag}(\Delta_1,\Delta_2,\ldots,\Delta_r)\;: \; \Delta_i\in\C^{n_i,n_i},~\sum_{i=1}^r n_i =n\right\}$ was addressed  in~\cite{doyle82} and shown that 
	\begin{equation}\label{doyle}
	\sup_{U} \rho(UM) = \mu_{\tilde {\mathcal S}}(M) \leq \inf_{D} \sigma_{\max}(DMD^{-1}),
\end{equation}
where $U=\text{diag}(U_1,\ldots,U_r)$ such that $U_i\in \C^{n_i,n_i}$ are unitary matrices,  $D=\text{diag}(d_1I_{n_1},\ldots,d_rI_{n_r})$ with $d_i>0$ for all 
$i$, and $\rho(\cdot)$ is  the spectral radius.
Moreover, if $r\leq 3$, then the inequality in~\eqref{doyle} is equality.
However, the framework in~\cite{doyle82} uses the unitary matrices, that can not be directly applied to the $\mu$-value problem subject to $\mathcal S$~\eqref{eq:defpermu}, when the digaonal blocks are rectangular. To extend the framework in~\cite{doyle82}, we use partial isometric matrices. 

We say that $P\in \C^{m,n}$ is a partially isometric matrix if for $x\in \C^n$ we have $\|Px\|=\|x\|$ whenever $x \perp \rm{null}(P)$. 
For given $x\in \C^m$ and $y\in \C^l\setminus\{0\}$, there exists a partially isometric matrix $P$ such that $Px=y$ if and only if $\|x\|=\|y\|$. Indeed,  if $\|x\|=\|y\|$, then $P=yx^*/\|x\|^2$ is a partially isometric matrix mapping $x$ to $y$. Conversely, for any partially isometric matrix $P$ satisfying $Px=y$, we have $\|x\|=\|Px\|=\|y\|$. 
We denote the set of block diagonal partial isometric matrices by $\mathbb{P}$ and is defined by
\begin{equation}
	\mathbb{P}:= \left\{\diag(P_1,P_2,\ldots, P_n) \; : \; P_i \in \C^{p_i,k_i}\; \text{are partially isometric matrices}, \; \sum_{i=1}^n k_i = k, \; \sum_{i=1}^n p_i = p \right\}.
\end{equation}
For $x=(x_1,\ldots,x_n)\in \R^n$, let us define two matrices $D_1(x)$ and $D_2(x)$ by
\begin{equation}\label{def:D1} 
	D_1(x):=\diag(\exp(x_1)I_{k_1},\exp(x_2)I_{k_2},\ldots,\exp(x_n)I_{k_n})
\end{equation}
and
\begin{equation}\label{def:D2}
	D_2(x):=\diag(\exp(x_1)I_{p_1},\exp(x_2)I_{p_2},\ldots,\exp(x_n)I_{p_n}),
\end{equation}
where $k_i$ and $p_i$ are such that $\sum_{i=1}^n k_i=k$ and $\sum_{i=1}^np_i=p$. 
Notice that for any $x\in \R^{n}$ and $\Delta\in\mathcal{S}$, we have $D_1(x)\Delta D_2(-x)=\Delta$ and hence $\mu_{\mathcal{S}}(M)=\mu_{\mathcal{S}}(D_2(-x)M D_1(x))$. 

Next we give the results obtained for the structured $\mu$-value, $\mu_{\mathcal{S}}(M)$. As the techniques used to obtain these results are similar to the ones used in~\cite{doyle82}, we ommit the proofs here and keep the details in the 
Appendix. 
\begin{theorem}\label{thm:formula1}
	Let $M\in \C^{k,p}$. Then the $\mu$-value of $M$ with respect to perturbations from the set $\mathcal{S}$ defined in~\eqref{eq:defpermu} is given by
	\begin{equation}\label{mu:lbound}
		\mu_{\mathcal{S}}(M)=\sup_{P\in\mathbb{P}} \rho (PM).
	\end{equation}
\end{theorem}
\begin{proof}
	See~\ref{Appendix1}.
\end{proof}	

\begin{theorem}\label{theorem:mu}
	Let $M\in \C^{k, p}$, and $D_1(x)$ and $D_2(x)$ be as defined in~\eqref{def:D1} and~\eqref{def:D2}, respectively. Then 
	the structured $\mu$-value of $M$ satisfies
	\begin{equation}\label{mu:ubound}
		\mu_{\mathcal{S}}(M)\leq \inf_{x\in \R^n} \sigma_{\max}\left( D_1(x)MD_2(-x)\right).
	\end{equation}
	Moreover, 
		\begin{enumerate}
	\item if the infimum	of $\sigma_{\max}(D_1(x)MD_2(-x))$ in~\eqref{mu:ubound} is attained at a point $\hat x\in \R^n$ and the optimal $\sigma_{\max}(D_1(\hat x)MD_2(-\hat x))$ is a simple singular value of $\hat{M}:=D_1(\hat x)MD_2(-\hat x)$, then equality holds in~\eqref{mu:ubound}.
	\item if the number of blocks in the pertubation matrix $\Delta$ in $\mathcal S$ is less than or equal to three, i.e., $n\leq 3$, then equality holds in~\eqref{mu:ubound}.
		\end{enumerate}
\end{theorem}
\begin{proof}
	See~\ref{Appendix2}.
\end{proof}	

By applying Theorem~\ref{theorem:mu} on the results in Sections~\ref{subsec:full} and~\ref{sec:parperbb}, we obtain a computable lower bound for the eigenvalue backward error of $S(z)$ with respect to full and partial perturbations. In particular,
 when the Rosenbrock matrix $S(z)$ is of the form $S(z)=\mat{cc}A-zI_r & B\\C & A_0\rix$, i.e., when $P(z)=A_0$, we have exact computable formulas for the eigenvalue backward errors in all cases of partial perturbations, as in such cases (Section~\ref{sec:parperbb}) the number of blocks in the perturbation matrix $\Delta$ in $\mathcal S$ would be less than or equal to three.

\section{Numerical experiments}\label{sec:num}
In this section, we consider some examples from literature to illustrate the results derived in this paper. We consider the rational eigenvalue problem and compute the corresponding Rosenbrock system matrix. 
Since the eigenvalue backward error of the Rosenbrock system matrix was reduced to an equivalent problem of computing $\mu$-values of some constant matrix, we compare our results with the most widely used MATLAB command ``mussv" to determine $\mu$-values.

We used GlobalSearch in MATLAB version 9.5.0 (R2018b) to compute the values of the associated optimization problem~(Theorem~\ref{theorem:mu}). 

\begin{example}[Vibration of fluid-solid structure~\cite{SuB11,Vos03}]{\rm
The rational eigenvalue problem arising from the simulation of mechanical vibration of fluid-solid structures is of the form
\begin{equation}\label{example1:eq1}
\left(M-zK + \sum_{i=1}^k \frac{z}{z-\alpha_i}E_i \right)x=0,
\end{equation}
where $\alpha_i>0$ are the poles for $i=1,2,\ldots,k$, $M,K\in \C^{n,n}$ are symmetric positive definite, and $E_i=C_iC_i^T$ with $C_i\in\C^{n,r_i} $ has rank $r_i$ for $i=1,2,\ldots,k$. It is easy to see that~\eqref{example1:eq1} can be written as
\begin{equation}
\left(M-zK-\sum_{i=1}^kC_iC_i^T - \sum_{i=1}^k \frac{\alpha_i}{\alpha_i-z}C_iC_i^T \right)x=0.
\end{equation} 
Let $\tilde{C}=[C_1\; C_2\; \ldots \; C_k]$ and $A={\diag}(\alpha_1I_{r_1},\alpha_2I_{r_2},\ldots,\alpha_kI_{r_k})$, then
\begin{equation}
\left(M-zK-\tilde{C}\tilde{C}^T - \tilde{C}A(A-zI_n)^{-1}\tilde{C}^T\right)x=0.
\end{equation}
The Rosenbrock system matrix in minimal state space form associated with the above rational eigenvalue problem is given by
\begin{equation}
S(z)=\mat{cc}A-zI_n & B\\ C & A_0+zA_1 \rix,
\end{equation}
where $B=\tilde{C}^T$, $C=\tilde{C}A$, $A_0=M-\tilde{C}\tilde{C}^T$ and $A_1=-K$. 
We randomly chose small scale matrices  $M,K\in\C^{5,5}$, $C\in \C^{5,3}$, and $A\in \C^{3,3}$ using the Matlab command randn. The eigenvalues of $S(z)$ are $-2.0575 - 0.3264i$, $0.9015 - 1.1390i$, $-0.2576 + 0.3492i$, $1.0829 + 0.7591i$, $1.0327 + 0.3282i$, $0.7219 + 0.1182i$, $0.2365 - 0.1340i$,  and $0.2357 - 0.9417i$. We take $\lambda= 0.7$, which is close to the eigenvalue $0.7219 + 0.1182i$ and compute the eigenvalue backward error $\eta^{\mathbb{S}}(\lambda,A,B,C,P)$  under perturbations to all blocks $A$, $B$, $C$, $A_0$ and $A_1$ of $S(z)$.  In view of~Theorem~\ref{theorem:ABCP}, the eigenvalue backward error is given by $\eta^{\mathbb{S}}(\lambda,A,B,C,P)= (\mu_{\mathcal S}(M))^{-1}$, where 
\[
M=\mat{cc}I_3 & \\ & I_5 \\I_3 & \\ & I_5 \\ & I_5 \rix (S(\lambda))^{-1} [I_{8}\; \lambda I_8]\mat{ccccc}I_3&I_3&&&\\&&I_5&I_5&\\&&&&0_{3,5}\\&&&&I_5 \rix
\]
 and $\mathcal{S}=\{{\diag}(\Delta_1,\Delta_2,\Delta_3,\Delta_4,\Delta_5)\;: \; \Delta_1\in\C^{3,3},\,\Delta_2\in \C^{3,5},\, \Delta_3\in \C^{5,3},\, \Delta_4\in \C^{5,5},\,\Delta_5\in\C^{5,5}\}$.
 Using Matlab command~\emph{mussv} with block structure given in $\mathcal S$, we obtain $20.485460811\leq\mu_{\mathcal{S}}(M)\leq 20.576426689$, and on applying Theorem~\ref{theorem:mu}, we get the upper bound $\mu_{\mathcal{S}}(M)\leq 20.485460812$, which is identical to the lower bound obtained using \emph{mussv} upto 9 digits. Thus, the structured eigenvalue backward error is given by  $\eta^{\mathbb{S}}(\lambda,A,B,C,P)=0.0488$. 
 
 If we consider perturbation only to blocks $B$ and $C$ of $S(z)$, then the eigenvalue backward error $\eta^{\mathbb{S}}(\lambda,B,C)$ is given by Theorem~\ref{theorem:two} as $\eta^{\mathbb{S}}(\lambda,B,C)=(\mu_{\mathcal{S}}(M))^{-1}$, where 
\[
M=\mat{cc}&I_5\\I_3& \rix S(\lambda)^{-1} \quad \text{and} \quad \mathcal{S}=\{{\diag}(\Delta_1,\Delta_2)\;:\; \Delta_1\in\C^{3,5},\;\Delta_2\in\C^{5,3}\}.
\]
Again using~\emph{mussv} command in MATLAB, the upper and lower bounds obtained for $\mu_{\mathcal{S}}(M)$ are $9.283991269267842$ and $9.272451127605631$, respectively. 
 However, when applying Theorem~\ref{theorem:mu}, the upper bound obtained is $9.272451127605656$ and is identical to the lower bound obtained using \emph{mussv} upto 13 decimal places.
Therefore, the upper bound provides the exact value for $\mu_{\mathcal{S}}(M)$.
Hence, we can deduce that $\eta^{\mathbb{S}}(\lambda,B,C)=0.10784$ gives the eigenvalue backward error.
%
%
}
\end{example}

\begin{example} {\rm \cite{guglielmi2017novel}
In the second example, we consider a matrix $M\in \C^{5,5}$ as 
\begin{equation}
	M=\mat{ccccc}
	i & \frac{1}{2}-\frac{1}{2}i & 1 & 1 & \frac{1}{2}\\ 
	\frac{1}{2} & -\frac{1}{2} & i & i & \frac{1}{2}-\frac{1}{2}i\\
	i & 1- \frac{1}{2} i & 1 & \frac{1}{2} & 0\\
	-\frac{1}{2} & \frac{1}{2}+i & -\frac{1}{2} +\frac{1}{2}i & 1+\frac{1}{2}i & \frac{1}{2}-\frac{1}{2}i\\
	\frac{1}{2}+i & \frac{1}{2}+\frac{1}{2}i & 0 & -\frac{1}{2} -\frac{1}{2}i & \frac{1}{2} - \frac{1}{2}i
	\rix
\end{equation}
and compute $\mu_{\mathcal{S}}(M)$ while considering perturbations from the set $\mathcal{S}=\{{\diag}(\Delta_1,\Delta_2)\; : \; \Delta_1\in \C^{2,3}, \Delta_2\in \C^{3,2}\}$.
Using the MATLAB function \emph{mussv} while considering the block structure as $[2,3;3,2]$, we obtain the bounds as $ 3.081980  \leq\mu_{\mathcal{S}}(M)\leq 3.097070$. However, by Theorem~\ref{theorem:mu}, we get the upper bound $\mu_{\mathcal{S}}(M)\leq 3.081980$, which is equal to the lower bound obtained using command mussv, showing the exactness of the upper bound in Theorem~\ref{theorem:mu}.
}
\end{example}

\section{Conclusion}

We have presented a reformulation for the structured eigenvalue backward error of the Rosenbrock system matrix in terms of $\mu$-value of a rectangular matrix with respect to rectangular block diagonal structure. For the reformulated problem, we have shown how to extend the results in~\cite{doyle82}. As a result, we obtain a lower bound for the structured eigenvalue backward error. Additionally, by exploiting the convexity of the joint numerical range , we obtain exact computable formulas for the  structured eigenvalue backward error for all cases of partial perturbations in the blocks of the Rosenbrock matrix. 

\bibliographystyle{plain}
\bibliography{PraS23a.bib}

@inproceedings{doyle82,
	title={Analysis of feedback systems with structured uncertainties},
	author={Doyle, John},
	booktitle={IEE Proceedings D (Control Theory and Applications)},
	volume={129},
	pages={242--250},
	year={1982},
	organization={IET Digital Library}
}

@BOOK{ZhoDG96,
	title = {Robust and Optimal Control},
	publisher = {Prentice Hall},
	year = {1996},
	author = {Zhou, K. and Doyle, J. C. and Glover, K.},
	address = {Upper Saddle River, NJ}
}

@article{Kar11,
	author = {Karow, M.},
	title = {$\mu$-Values and Spectral Value Sets for Linear Perturbation Classes Defined by a Scalar Product},
	journal = {SIAM J. Matrix Anal. Appl.},
	volume = {32},
	number = {3},
	pages = {845-865},
	year = {2011}
}

@phdthesis{Kar03,
	author = {Karow, M.},
	title = {Geometry of spectral value sets},
	publisher = {University of Bremen, Bremen, Germany},
	address = {University of Bremen, Bremen, Germany},
	year = {2003}
}

@article{guglielmi2017novel,
  title={A novel iterative method to approximate structured singular values},
  author={Guglielmi, Nicola and Rehman, Mutti-Ur and Kressner, Daniel},
  journal={SIAM J. Matrix Anal. Appl.},
  volume={38},
  number={2},
  pages={361--386},
  year={2017},
  publisher={SIAM}
}

@article{PraS22b,
	title={Structured eigenvalue backward errors for rational matrix functions with symmetry structures},
	author={Prajapati, Anshul and Sharma, Punit},
	journal={BIT Numerical Mathematics},
	volume={64},
	number={1},
	pages={10},
	year={2024},
	publisher={Springer}
}

@article{gracia2015directional,
  title={Directional derivatives of the singular values of matrices depending on several real parameters},
  author={Gracia, Juan-Miguel},
  journal={arXiv preprint arXiv:1504.01679},
  year={2015}
}

@article{trenkler2004matrices,
  title={Matrices which take a given vector into a given vector—revisited},
  author={Trenkler, G{\"o}tz},
  journal={The American Mathematical Monthly},
  volume={111},
  number={1},
  pages={50--52},
  year={2004},
  publisher={Taylor \& Francis}
}

@article {SuB11,
    AUTHOR = {Su, Yangfeng and Bai, Zhaojun},
     TITLE = {Solving rational eigenvalue problems via linearization},
   JOURNAL = {SIAM J. Matrix Anal. Appl.},
  FJOURNAL = {SIAM Journal on Matrix Analysis and Applications},
    VOLUME = {32},
      YEAR = {2011},
    NUMBER = {1},
     PAGES = {201--216},
      ISSN = {0895-4798},
   MRCLASS = {65F15 (15A18 65H17)},
  MRNUMBER = {2811297},
MRREVIEWER = {Fernando De Ter\'{a}n},
       DOI = {10.1137/090777542},
       URL = {https://doi.org/10.1137/090777542},
}

@article {MR2496422,
    AUTHOR = {Adhikari, Bibhas and Alam, Rafikul},
     TITLE = {Structured backward errors and pseudospectra of structured
              matrix pencils},
   JOURNAL = {SIAM J. Matrix Anal. Appl.},
  FJOURNAL = {SIAM Journal on Matrix Analysis and Applications},
    VOLUME = {31},
      YEAR = {2009},
    NUMBER = {2},
     PAGES = {331--359},
      ISSN = {0895-4798},
   MRCLASS = {65F15 (65F35)},
  MRNUMBER = {2496422},
MRREVIEWER = {Raf Vandebril},
       DOI = {10.1137/070696866},
       URL = {https://doi.org/10.1137/070696866},
}

@article {MR2780396,
    AUTHOR = {Adhikari, Bibhas and Alam, Rafikul},
     TITLE = {On backward errors of structured polynomial eigenproblems
              solved by structure preserving linearizations},
   JOURNAL = {Linear Algebra Appl.},
  FJOURNAL = {Linear Algebra and its Applications},
    VOLUME = {434},
      YEAR = {2011},
    NUMBER = {9},
     PAGES = {1989--2017}
}

@article {MR3194659,
    AUTHOR = {Bora, Shreemayee and Karow, Michael and Mehl, Christian and
              Sharma, Punit},
     TITLE = {Structured eigenvalue backward errors of matrix pencils and
              polynomials with {H}ermitian and related structures},
   JOURNAL = {SIAM J. Matrix Anal. Appl.},
  FJOURNAL = {SIAM Journal on Matrix Analysis and Applications},
    VOLUME = {35},
      YEAR = {2014},
    NUMBER = {2},
     PAGES = {453--475},
      ISSN = {0895-4798},
   MRCLASS = {15A22 (65F15 65F35)},
  MRNUMBER = {3194659},
MRREVIEWER = {Jing Li},
       DOI = {10.1137/130925621},
       URL = {https://doi.org/10.1137/130925621},
}

@article {MR3335496,
    AUTHOR = {Bora, Shreemayee and Karow, Michael and Mehl, Christian and
              Sharma, Punit},
     TITLE = {Structured eigenvalue backward errors of matrix pencils and
              polynomials with palindromic structures},
   JOURNAL = {SIAM J. Matrix Anal. Appl.},
  FJOURNAL = {SIAM Journal on Matrix Analysis and Applications},
    VOLUME = {36},
      YEAR = {2015},
    NUMBER = {2},
     PAGES = {393--416},
      ISSN = {0895-4798},
   MRCLASS = {15A22 (15A18 15A60 47A56 65F15 65F35 93C73)},
  MRNUMBER = {3335496},
MRREVIEWER = {Ivica Naki\'{c}},
       DOI = {10.1137/140973839},
       URL = {https://doi.org/10.1137/140973839},
}

@article {MR4404572,
    AUTHOR = {Prajapati, Anshul and Sharma, Punit},
     TITLE = {Optimizing the {R}ayleigh quotient with symmetric constraints
              and its application to perturbations of structured polynomial
              eigenvalue problems},
   JOURNAL = {Linear Algebra Appl.},
  FJOURNAL = {Linear Algebra and its Applications},
    VOLUME = {645},
      YEAR = {2022},
     PAGES = {256--277},
      ISSN = {0024-3795},
   MRCLASS = {15A18 (15A22 65K05 93C05 93C73)},
  MRNUMBER = {4404572},
       DOI = {10.1016/j.laa.2022.03.016},
       URL = {https://doi.org/10.1016/j.laa.2022.03.016},
}

@article{Tis00,
title = {Backward error and condition of polynomial eigenvalue problems},
author = {Françoise Tisseur},
journal = {Linear Algebra Appl.},
volume = {309},
number = {1},
pages = {339-361},
year = {2000}
}

@article {MR4098788,
    AUTHOR = {Ahmad, Sk. Safique and Kanhya, Prince},
     TITLE = {Structured perturbation analysis of sparse matrix pencils with
              {$s$}-specified eigenpairs},
   JOURNAL = {Linear Algebra Appl.},
  FJOURNAL = {Linear Algebra and its Applications},
    VOLUME = {602},
      YEAR = {2020},
     PAGES = {93--119},
      ISSN = {0024-3795},
   MRCLASS = {65F15 (15A12 15A18 65F35)},
  MRNUMBER = {4098788},
MRREVIEWER = {Kang Zhao},
       DOI = {10.1016/j.laa.2020.04.030},
       URL = {https://doi.org/10.1016/j.laa.2020.04.030},
}

@article {Vos03,
	AUTHOR = {Voss, Heinrich},
	TITLE = {A rational spectral problem in fluid-solid vibration},
	JOURNAL = {Electronic Transactions on Numerical Analysis},
	VOLUME = {16},
	YEAR = {2003},
	PAGES = {93--105},
	ISSN = {1068-9613},
	MRCLASS = {74H45 (49R50 65H17 74F10)},
	MRNUMBER = {1988722},
	MRREVIEWER = {Jan\ K\v{r}\'{\i}\v{z}},
}

@book{HorJ85,
	author    = { Horn, R. A. and   Johnson, C. R.},
	title     = {Matrix Analysis},
	publisher = {Cambridge University Press},
	address   = {Cambridge},
	year      = {1985}
}

@article{AhmA09,
	title={Pseudospectra, critical points and multiple eigenvalues of matrix polynomials},
	author={Ahmad, Sk Safique and Alam, Rafikul},
	journal={Linear algebra and its applications},
	volume={430},
	number={4},
	pages={1171--1195},
	year={2009},
	publisher={Elsevier}
}

@article{Ros70,
	title={State Space and Multivariable Theory, Thomas Nelson and Sons},
	author={Rosenbrock, HH},
	journal={Ltd., London},
	year={1970}
}

@article{LuPSB25,
	title={Eigenvalue Backward Errors of Rosenbrock Systems and Optimization of Sums of Rayleigh Quotients},
	author={Lu, Ding and Prajapati, Anshul and Sharma, Punit and Bora, Shreemayee},
	journal={SIAM Journal on Matrix Analysis and Applications},
	volume={46},
	number={2},
	pages={1301--1327},
	year={2025},
	publisher={SIAM}
}

\appendix 

\section{Proof of Theorem~\ref{thm:formula1}}\label{Appendix1}

We first state a result from~\cite{doyle82}, that will be useful in proving Theorem~\ref{thm:formula1}.
\begin{lemma}{\rm \cite{doyle82}}\label{lemma:doyle}
	Let $p(z_1,z_2,\ldots,z_n)$ be a polynomial in $n$ complex variables, and let $(\hat{z_1},\hat{z_2},\ldots, \hat{z_n})$ be such that it is a solution of $p(z_1,z_2,\ldots,z_n)=0$ with minimum $\|\cdot\|_{\infty}$ norm, where $\|\cdot\|_{\infty}$ stands for the $\infty$-norm of a vector. If $(\hat{z_1},\hat{z_2},\ldots, \hat{z_n})$ is real and non-negative, then there exists $(x_1,x_2,\ldots,x_n) \in \R^n$ non-negative such that $p(x_1,x_2,\ldots,x_n)=0$ and $x_i=\|\hat{z}\|_{\infty}$ for all $i=1,\ldots,n$.
\end{lemma}

\noindent	{\underline{\it Proof of Theorem~\ref{thm:formula1}}}:~
	From~\eqref{def:mu}, we have
	\begin{equation*}
		\mu_{\mathcal S}(M)= \left(\min \left\{ \|\Delta\|\; : \; \Delta \in \mathcal{S}, \; \text{det}(I_p - \Delta M) = 0 \right\} \right)^{-1}.
	\end{equation*}
	Note that $\mathbb P \subseteq \mathcal S$, and for any $\lambda \in \C \setminus \{0\}$ and fixed $P \in \mathbb P$, $\lambda P \in \mathbb S$. This implies that 
	\begin{eqnarray}
		\mu_{\mathcal S}(M)
		&\geq& \left(\min \left\{ \|\lambda P\|\; : \; \lambda \in \C \setminus \{0\}, \; \text{det}(I_p - \lambda P M) = 0 \right\} \right)^{-1}\nonumber\\
		&=&\left(\min\left\{|\lambda|\; : \; \lambda \in \C \setminus \{0\},\; \text{det}(I_p- \lambda PM)=0 \right\}\right)^{-1}\label{eq:mures1} \\
		&=&\rho (PM), \label{eq:mures2} 
	\end{eqnarray}
	where~\eqref{eq:mures1}  follows due to the fact that $\|P\|=1$, since $P$ is partially isometric. The inequality~\eqref{eq:mures2} holds for any partially isometric matrix $P \in \mathbb P$, this implies that  
	\begin{equation}\label{eq:mures3} 
		\mu_{\mathcal S}(M) \geq \sup_{P \in \mathbb P} \rho (PM).
	\end{equation}
	To show equality holds in~\eqref{eq:mures3}, let $\mu_{\mathcal S}(M)= \frac{1}{\delta}$ for some $\delta > 0$, since $\mu_{\mathcal S}(M) \geq 0$ (if $\mu_{\mathcal S}(M)=0$, then there is nothing to prove). Then there exists $\hat \Delta \in \mathbb S$ such that 
	\begin{equation}\label{eq:mures4} 
		\|\hat \Delta \| = \delta \quad \text{and} \quad \text{det}(I_p-\hat \Delta M)=0.
	\end{equation}
	Let $\hat \Delta=\diag(\hat \Delta_1,\ldots,\hat \Delta_n)$, and for each $j=1,\ldots,n$ let $\hat \Delta_j=U_j\Sigma_jV_j^*$ be an SVD of $\hat \Delta_j$, where 
	$U_j\in \C^{p_j, k_j}$, $V_j\in \C^{k_j, k_j}$, and $\Sigma_j \in \C^{k_j \times k_j}$ is the diagonal matrix containing the singular values of $\hat \Delta_j$. Consider $\hat U=\diag(U_1,\ldots,U_n)\in \C^{p, k}$, $\hat V=\diag(V_1,\ldots,V_n)\in \C^{k, k}$, and $\hat \Sigma=\diag(\Sigma_1,\ldots,\Sigma_n)\in \C^{k, k}$. Then the determinant condition in~\eqref{eq:mures4} can be written as 
	\begin{equation*}
		0 =\rm{det}(I_p-\hat \Delta M)=\rm{det}(I_p-\hat U\hat \Sigma \hat V^* M).         
	\end{equation*}
	Thus, in view of Lemma~\ref{lemma:doyle}, the diagonal matrix $X=\hat \Sigma$ is a solution of the polynomial equation $\text{det}(I_p-\hat UX\hat V^*M)=0$ with $X\in \C^{k \times k}$ diagonal, of minimal norm $\|\hat \Sigma\|=\delta=\|\hat \Delta\|$. Indeed, if there exists another solution  $\tilde \Sigma \in \C^{k \times k}$ with $\|\tilde \Sigma \|< \|\hat \Sigma \|=\delta$, then the matrix $\tilde \Delta=\hat U \tilde \Sigma \hat V^*$ satisfies that
	\[
	\|\tilde \Delta\|=\|\tilde \Sigma\|< \|\hat \Sigma\|=\delta \quad \text{and}\quad
	\text{det}(I_p-\tilde \Delta M)=0,
	\]
	which is a contradiction of the fact that $\mu_{\mathcal S}(M)=\frac{1}{\delta}$. Since $\hat \Sigma$ is real and nonnegative, by Lemma~\ref{lemma:doyle} we have that 
	\begin{equation}
		0=\text{det}(I_p-\hat U (\delta I_k)\hat V^*M)=\text{det}(I_p-\delta \hat P M),
	\end{equation}
	where $\hat P=\hat U \hat V^*$ is a partially isometric matrix. This implies that 
	\[
	\rho(PM) \geq \frac{1}{\delta}=\mu_{\mathcal S}(M),
	\]
	which in combination with~\eqref{eq:mures3} proves the equality in~\eqref{eq:mures3}.


\section{Proof of Theorem~\ref{theorem:mu}}\label{Appendix2}

Let the SVD of $M \in \C^{k, p}$ be given by
\begin{equation}\label{eq:svdM}
	M=\sigma_{\max}(M)U_1V_1^* + U_2\Sigma V_2^*,
\end{equation}
where $[U_1\; U_2]$ and $[V_1\; V_2]$ are unitary matrices with $U_1 \in \C^{k, r}$ and $ V_1\in \C^{p, r}$, where $r$ is the multiplicity of the largest singular value $\sigma_{\max}(M)$. For a given $x=(x_1,\ldots,x_n) \in \R^n$, let $M(x)$ be defined by
\begin{equation}
	M(x):=D_1(x)MD_2(-x),
\end{equation}
where $D_1(x)$ and $D_2(x)$ are defined by~\eqref{def:D1} and~\eqref{def:D2} , respectively. 
Then the directional derivative of $\sigma_{\max}(M(x))$ is a linear map $D\sigma_{\max}(M(x)):\R^n\rightarrow \C$ such that 
\[
\lim_{\|h\|\rightarrow 0} \frac{|\sigma_{\max}(M(x+h)) - \sigma_{\max}(M(x))-D\sigma_{\max}(M(x))h|}{\|h\|} = 0,
\]
and in  a neighborhood of zero, we can write $\sigma_{\max}(M(x))$ approximately as 
\begin{equation}\label{sigmaM}
	\sigma_{\max}(M(x))=\sigma_{\max}(M) + D\sigma_{\max}(M(0))x + o(x),
\end{equation}
where $D\sigma_{\max}(M(0))$ denotes the directional derivative of $\sigma_{\max}(M(x))$  evaluated at the zero vector. By~\cite[Theorem 9]{gracia2015directional}, directional derivative of the largest singular value is given by
\begin{equation}\label{der:sigmaM}
	D\sigma_{\max}(M(0))x= \lambda_{\max}\left(\sum_{i=1}^n x_i \Re\left(U_1^*\left(\frac{\partial}{\partial x_i}M(0)\right) V_1\right)\right),
\end{equation}
where $\Re(.)$ stands for the real part.
We know that the partial derivatives at zero depend only on the first order terms of $M(x)$, thus by writing $M(x)$ only upto first order terms, we get  
\begin{equation}\label{M:expand}
	M(x)=D_1(x)MD_2(-x)=M + \sum_{i=1}^nx_iM_i,
\end{equation}
where $M_i$'s for $i=1,\ldots, n$ are given by
\begin{equation}\label{M:Mi}
	M_i=\mat{ccccc} 0_{k_1} & & & &\\ & \ddots & & &\\ & & I_{k_i} & &\\ & & & \ddots & \\ & & & & 0_{k_n} \rix M - M \mat{ccccc} 0_{p_1} & & & &\\ & \ddots & & &\\ & & I_{p_i} & &\\ & & & \ddots & \\ & & & & 0_{p_n} \rix .
\end{equation}
Thus using~\eqref{M:expand} and~\eqref{der:sigmaM} in~\eqref{sigmaM}, we obtain
\begin{equation} \label{eq:temp1sig}
	\sigma_{\max}(M(x)) = \sigma_{\max}(M) + \lambda_{\max}(\sum_{i=1}^n x_i H_i) + o(x),
\end{equation}
where
\begin{align}\label{def:Hi}
	H_i=\real(U_1^*M_iV_1)=\frac{1}{2} (U_1^*M_i V_1 + V_1^*M_iU_1).
\end{align}
Let us write $U_1\in \C^{k, r}$ and $V_1\in \C^{p, r}$ as
\begin{align}\label{eq:temp12}
	U_1=\mat{c}\alpha_1\\ \alpha_2 \\ \vdots \\ \alpha_n  \rix \quad \text{and} \quad  V_1=\mat{c}\beta_1\\ \beta_2 \\ \vdots \\ \beta_n  \rix,
\end{align}
where $\alpha_i\in \C^{k_i, r}$ and $\beta_i\in \C^{p_i, r}$. 
%
This implies that, for each $i=1,\ldots,n$, we have
\begin{align}\label{def:Hi2}
	U_1^*M_iV_1&=\sigma_{\max}(M)\mat{cccc}\alpha_1^*& \alpha_2^*& \cdots & \alpha_n^* \rix \mat{ccccc}& & -\alpha_1\beta_i^* & &\\ & & -\alpha_2\beta_i^* & & \\ & & \vdots & & \\ \alpha_i\beta_1^* & \cdots & 0 & \cdots & \alpha_i\beta_n^* \\ & & \vdots & &\\ & & -\alpha_n\beta_i^* & & \rix \mat{c}\beta_1\\ \beta_2 \\ \vdots \\ \\ \\ \beta_n \rix \nonumber \\
	&=\sigma_{\max}(M)(\alpha_i^*\alpha_i\beta_1^*\beta_1+\alpha_i^*\alpha_i\beta_2^*\beta_2 + \cdots + \alpha_i^*\alpha_i\beta_{i-1}^*\beta_{i-1}+\alpha_i^*\alpha_i\beta_{i+1}^*\beta_{i+1}+\cdots+\alpha_i^*\alpha_i\beta_n^*\beta_n \nonumber \\ & \quad -\alpha_1^*\alpha_1\beta_i^*\beta_i -\alpha_2^*\alpha_2\beta_i^*\beta_i - \cdots -\alpha_{i-1}^*\alpha_{i-1}\beta_i^*\beta_i-\alpha_{i+1}^*\alpha_{i+1}\beta_i^*\beta_i - \cdots -\alpha_n^*\alpha_n\beta_i^*\beta_i)\nonumber \\
	& = \sigma_{\max}(M)(\alpha_i^*\alpha_i(I_r-\beta_i^*\beta_i) - (I_r-\alpha_i^*\alpha_i)\beta_i^* \beta_i) \nonumber \\
	& = \sigma_{\max}(M)(\alpha_i^*\alpha_i - \beta_i^*\beta_i).
\end{align}
Using~\eqref{def:Hi2} in~\eqref{def:Hi}, we obtain
\begin{equation}\label{def:Hi3}
	H_i=\sigma_{\max}(M)\cdot(\alpha_i^*\alpha_i - \beta_i^*\beta_i) \quad \text{for}~i=1,\ldots,n,
\end{equation}
since $(U_1^*M_iV_1)^* = U_1^*M_iV_1$.
We are now ready to prove Theorem~\ref{theorem:mu}. 
%

	\noindent	{\underline{\it Proof of Theorem~\ref{theorem:mu}}}:~
	For any $x\in \R^n$ and $\Delta \in \mathbb{S}$, we have $D_2(-x)\Delta D_1(x)=\Delta$.
	This implies from~\eqref{def:mu} that 
	\begin{align}
		\mu_{\mathcal{S}}(M)&=\left(\inf\left\{\sigma_{\max}(\Delta)\; : \; \Delta \in \mathbb{S}, \; \text{det}(I_p-D_2(-x) \Delta D_1(x)M)=0 \right\} \right)^{-1} \nonumber\\
		& \leq \left(\inf\left\{\sigma_{\max}(\Delta)\; : \; \Delta \in \C^{p, k}, \; \text{det}(I_p-\Delta D_1(x)MD_2(-x) )=0 \right\} \right)^{-1} \nonumber \\
		& = \sigma_{\max}(D_1(x)MD_2(-x)).
	\end{align}
	Thus, we have
	\begin{equation*}
		\mu_{\mathcal{S}}(M)\leq \inf_{x\in \R^n}\sigma_{\max}\left(D_1(x)MD_2(-x)\right).
	\end{equation*}
Let us first show that if $0\in \mathcal{W}(H_1,\ldots,H_n)$, where $\mathcal{W}(H_1,\ldots,H_n)$ is the joint numerical range of matrices $H_1,\ldots,H_n$ defined in~\eqref{def:Hi3}, then we have 
\begin{equation}\label{mu:ubound11}
	\mu_{\mathcal{S}}(M)= \sigma_{\max}\left( D_1(0)MD_2(0)\right)=\sigma_{\max}(M).
\end{equation}
Indeed, if $0\in \mathcal{W}(H_1,\ldots,H_n)$, then there exists $v\in \C^r\setminus \{0\}$ such that $v^*H_iv=0$ for all $i=1,\ldots, n$, which implies from~\eqref{def:Hi3} that $\|\beta_iv\|=\|\alpha_iv\| $ for all $i=1,\ldots,n$. Thus, for each $i=1,\ldots,n$, there exists a partially isometric matrix $P_i\in \C^{p_i, k_i}$ such that 
	\begin{equation}\label{eq:temp11}
		\beta_iv = P_i\alpha_i v.
	\end{equation}
	By setting $\tilde P=\text{diag}(P_1,\ldots,P_n)$, we get that $\tilde P \in \mathbb{P}$ and in view of~\eqref{eq:temp11} and~\eqref{eq:temp12}, we have $V_1v=\tilde PU_1 v$. This implies that 
	$U_1v=U_1V_1^*\tilde PU_1v$. This shows that there exists a nonzero vector $y=U_1v \in  \C^{k, 1}$ such that $y=U_1V_1^*\tilde Py$. Using this in~\eqref{eq:svdM}, we have
	\begin{eqnarray}
		M\tilde Py&=&\left(\sigma_{\max}(M)U_1V_1^* + U_2\Sigma V_2^*\right)\tilde Py \nonumber\\
		&=& \sigma_{\max}(M)U_1V_1^* \tilde Py + U_2\Sigma V_2^* \tilde Py \nonumber \\
		&=& \sigma_{\max}(M)U_1V_1^* \tilde Py \quad \quad (\because \tilde Py=V_1v) \nonumber\\
		&=& \sigma_{\max}(M)y \quad \quad (\because y=U_1V_1^*\tilde Py) \label{eq:eq:y1vector}.
	\end{eqnarray}
	Also, since $\rho(MP)\leq \|M\|=\sigma_{\max}(M)$ for any $P \in \mathbb P$, we have that 
	\begin{equation}\label{eq:eq:y1vector1}
		\sigma_{\max}(M) \leq \rho(M\tilde P) \leq \sup_{P \in \mathbb P}\rho(MP)\leq \|M\|=\sigma_{\max}(M).
	\end{equation}
	Thus from~\eqref{eq:eq:y1vector1} and Theorem~\ref{thm:formula1}, we have that 
	\[
	\sup_{P\in \mathbb{P}}\rho(MP) = \mu_{\mathcal{S}}(M)=\sigma_{\max}(M),
	\]
	which proves~\eqref{mu:ubound11}.

{\it Proof of 1}.~	If $\sigma_{\max}(D_1(\hat x)MD_2(-\hat x))$ is a simple singular value of $\hat{M}:=D_1(\hat x)MD_2(-\hat x)$, then $\sigma_{\max}(D_1(x)MD_2(-x))$ is partially differentiable at $x=\hat x$ and in view of~\eqref{eq:temp1sig} we have
\begin{eqnarray*}
	0=\frac{\partial}{\partial x_i} \sigma_{\max}(D_1(x+\hat x)MD_2(-(x+\hat x))) \Big|_{x=0}= \frac{\partial}{\partial x_i}  \lambda_{\max}(\sum_{i=1}^n x_i \hat H_i) \Big|_{x=0}= 
	\hat{H_i},
\end{eqnarray*}
since $\lambda_{\max}(\sum_{i=1}^n x_i \hat H_i) =\sum_{i=1}^n x_i \hat H_i$,  where $\hat{H_i}$ are $1 \times 1$ matrices defined by~\eqref{def:Hi3} for the matrix $\hat{M}$. This implies that  $\mathcal{W}(\hat H_1,\ldots,\hat H_n) =0$ and hence, in view of~\eqref{mu:ubound11}, equality holds in~\eqref{mu:ubound}.

{\it Proof of 2}.~If the number of blocks in the pertubation matrix $\Delta$ is less than or equal to three, i.e., $n\leq 3$, then $\mathcal{W}(H_1,\ldots,H_n)$ is a convex set, where Hermitian matrices $H_1,\ldots,H_n$ are as defined in~\eqref{def:Hi3}.
This is due to the facts that (a) $H_1+H_2+H_3= 0$ and thus  $\mathcal{W}(H_1,H_2,H_3)$ depends only on two Hermitian matrices, and (b) the joint numerical range of any two Hermitian matrices in convex~\cite{HorJ85}.

		Thus, for $n\leq 3$, equality holds in~\eqref{mu:ubound}, because if any optimization algorithm is implemented, it can proceed until a local minima is found, equivalently, $0\in \mathcal{W}(H_1,\ldots,H_n)$. If suppose $0\not \in\mathcal{W}(H_1,\ldots,H_n)$, then by taking $x_0:=\text{argmin}\{\|x\|~:~x \in \mathcal{W}(H_1,\ldots,H_n)\}$ and $v_0$ an eigenvector corresponding to the smallest eigenvalue of $\sum_{i=1}^n x_{0_{i}}H_i$, we obtain
\begin{equation}\label{eq:remlam}
	\lambda_{\min}(\sum_{i=1}^n x_{0_{i}}H_i)=v_0^*(\sum_{i=1}^n x_{0_{i}}H_i)v_0 = \sum_{i=1}^n x_{0_{i}}v_0^*H_iv_0=\|x_0\|\|y_0\|cos(\theta)\geq \|x_0\|^2 >0,
\end{equation}
where $y_0:=(v_0^*H_1v_0,\ldots,v_0^*H_nv_0)$ and $\theta$ represents angle between $x_0$ and $y_0$. The last inequality in~\eqref{eq:remlam} follows because $\|y_0\|cos(\theta) \geq \|x_0\|$ as $\mathcal{W}(H_1,\ldots,H_n)$ is a convex set and $x_0:=\text{arg min}\{\|x\|~:~x \in \mathcal{W}(H_1,\ldots,H_n)\}$, which implies that the norm of the component of $y$ in the direction of $x_0$ will be larger than the norm of $x_0$. Hence, by~\eqref{eq:temp1sig} there exists $\epsilon_0>0$ such that for all $0<\epsilon< \epsilon_0$, we have
\begin{eqnarray*}
	\sigma_{\max}(M(-\epsilon x_0)) &=& \sigma_{\max}(M) + \lambda_{\max}(\sum_{i=1}^n - \epsilon {x_0}_i H_i) + o(-\epsilon x)\\ &=& \sigma_{\max}(M) - \epsilon \lambda_{\min}(\sum_{i=1}^n {x_0}_i H_i) + o(-\epsilon x) \\
	&<& \sigma_{\max}(M).
\end{eqnarray*}
Implying the possibility of further improvement and hence at the optimal we must have $0\in\mathcal{W}(H_1,\ldots,H_n)$. Hence, in view of~\eqref{mu:ubound11}, equality holds in~\eqref{mu:ubound}.
%

\end{document}